\documentclass[a4paper,12pt,reqno]{amsart}
\usepackage[T1]{fontenc}
\usepackage[utf8]{inputenc}
\usepackage{amssymb,dsfont,mathrsfs}
\usepackage[margin=1in]{geometry}
\usepackage{enumitem}
\usepackage[mathscr]{eucal}
\usepackage{graphicx}
\usepackage[bookmarksdepth=2]{hyperref}

\numberwithin{equation}{section}

\newtheorem{theorem}{Theorem}[section]
\newtheorem{lemma}[theorem]{Lemma}
\newtheorem{corollary}[theorem]{Corollary}
\newtheorem{proposition}[theorem]{Proposition}

\theoremstyle{definition}
\newtheorem{definition}[theorem]{Definition}

\newtheorem{example}[theorem]{Example}
\newtheorem{remark}[theorem]{Remark}

\DeclareMathOperator{\re}{Re}
\DeclareMathOperator{\im}{Im}

\DeclareMathOperator{\sign}{sign}

\DeclareMathOperator{\ind}{\mathds{1}}

\newcommand{\cm}{\mathscr{CM}}
\newcommand{\am}{\mathscr{AM}}
\newcommand{\amcm}{\am\text{-}\cm}

\newcommand{\laplace}{\mathscr{L}}

\newcommand{\C}{\mathds{C}}
\newcommand{\R}{\mathds{R}}
\newcommand{\Z}{\mathds{Z}}

\newcommand{\ph}{\varphi}
\newcommand{\eps}{\varepsilon}

\renewcommand{\le}{\leqslant}
\renewcommand{\ge}{\geqslant}

\usepackage{ifthen}
\newcommand{\formula}[2][nolabel]%
{%
 \ifthenelse{\equal{#1}{nolabel}}%
 {\begin{align*} \begin{aligned} #2 \end{aligned} \end{align*}}%
 {%
  \ifthenelse{\equal{#1}{}}%
  {\begin{align} #2 \end{align}}%
  {\begin{align} \label{#1} \begin{aligned} #2 \end{aligned} \end{align}}%
 }%
}

\begin{document}

\title[A new class of bell-shaped functions]{A new class of bell-shaped functions}
\author{Mateusz Kwaśnicki}
\thanks{Work supported by the Polish National Science Centre (NCN) grant no.\@ 2015/19/B/ST1/01457}
\address{Mateusz Kwaśnicki \\ Faculty of Pure and Applied Mathematics \\ Wrocław University of Science and Technology \\ ul. Wybrzeże Wyspiańskiego 27 \\ 50-370 Wrocław, Poland}
\email{mateusz.kwasnicki@pwr.edu.pl}
\dedicatory{In memory of Augustyn Kałuża}

\begin{abstract}
We provide a large class of functions $f$ that are bell-shaped: the $n$-th derivative of $f$ changes its sign exactly $n$ times. This class is described by means of Stieltjes-type representation of the logarithm of the Fourier transform of $f$, and it contains all previously known examples of bell-shaped functions, as well as all extended generalised gamma convolutions, including all density functions of stable distributions. The proof involves representation of $f$ as the convolution of a Pólya frequency function and a function which is absolutely monotone on $(-\infty, 0)$ and completely monotone on $(0, \infty)$. In the final part we disprove three plausible generalisations of our result.
\end{abstract}

\maketitle

%
%

\section{Introduction and main results}

By Rolle's theorem, for every $n = 0, 1, 2, \ldots$ the $n$-th derivative $f^{(n)}$ of any smooth function $f$ which converges to zero at $\pm \infty$ changes its sign at least $n$ times. Such $f$ is said to be bell-shaped if $f^{(n)}$ changes its sign exactly $n$ times. Note that in this case $f^{(n)}$ necessarily converges to zero at $\pm\infty$ for every $n = 1, 2, \ldots$\, The study of bell-shaped functions originated in the theory of games, see Section~6.11.C in~\cite{bib:k68}. The main result of~\cite{bib:g84} asserted that all stable distributions have bell-shaped density functions. However, the argument given there contained an error, which led to an open problem that remained unanswered for over 30 years. In the present paper we provide a correct proof of a much more general result.

It is elementary to prove that $e^{-x^2}$, $(1 + x^2)^{-p}$ and $x^{-p} e^{-1/x} \ind_{(0, \infty)}(x)$ are bell-shaped for $p > 0$. I.~I.~Hirschman proved in~\cite{bib:h50} that there are no compactly supported bell-shaped functions, thus resolving a conjecture of I.~J.~Schoenberg. A classical result due to the latter asserts that P\'olya frequency functions (which are discussed in Section~\ref{sec:polya}) are bell-shaped; see~\cite{bib:hw55}. T.~Simon proved in~\cite{bib:s15} that positive stable distributions have bell-shaped density functions. Density functions of general stable distributions are known to be bell-shaped when their index of stability is $2$ or $1/n$ for some $n = 1, 2, \ldots$: this follows from~\cite{bib:g84} after correcting an error in the proof, see~\cite{bib:s15} for a detailed discussion. W.~Jedidi and T.~Simon showed in~\cite{bib:js15} that hitting times of (generalised) diffusions are bell-shaped. The main purpose of this article is to provide a new class of bell-shaped functions which, to the best knowledge of the author, contains all known examples of bell-shaped functions.

We say that a locally integrable function (or, more generally, a locally finite measure) $f$ is weakly bell-shaped if the convolution of $f$ with the Gauss--Weierstrass kernel $(4 \pi t)^{-1/2} e^{-x^2 / (4 t)}$ is bell-shaped for every $t > 0$.

\begin{theorem}
\label{thm:bell}
Suppose that $f$ is a locally integrable function which converges to zero at $\pm\infty$, and which is decreasing near $\infty$ and increasing near $-\infty$. Suppose furthermore that for $\xi \in \R \setminus \{0\}$ the Fourier transform of $f$ satisfies
\formula[bell]{
 \laplace f(i \xi) &:= \int_{-\infty}^\infty e^{-i \xi x} f(x) dx \\
 & \phantom{:}= \exp\biggl(-a \xi^2 - i b \xi + c + \int_{-\infty}^\infty \biggl( \frac{1}{i \xi + s} - \biggl(\frac{1}{s} - \frac{i \xi}{s^2} \biggr) \ind_{\R \setminus (-1, 1)}(s) \biggr) \ph(s) ds \biggr)
}
(with the former integral understood as an improper integral); here $a \ge 0$, $b, c \in \R$ and $\ph : \R \to \R$ is a function with the following properties:
\begin{enumerate}[label=\textnormal{(\alph*)}]
\item\label{thm:bell:a} for every $k \in \Z$ the function $\ph(s) - k$ changes its sign at most once, and for $k = 0$ this change takes place at $s = 0$: we have $\ph(s) \ge 0$ for $s > 0$ and $\ph(s) \le 0$ for $s < 0$;
\item\label{thm:bell:c} we have
\formula{
 \biggl( \int_{-\infty}^{-1} + \int_1^\infty \biggr) \frac{|\ph(s)|}{|s|^3} \, ds < \infty ;
}
\item\label{thm:bell:d} we have
\formula{
 \int_{-1}^1 \re \laplace f(i \xi) d\xi & < \infty &\qquad& \text{and} &\qquad \lim_{\xi \to 0} \im \laplace f(i \xi) & = 0 .
}
\end{enumerate}
Then $f$ is weakly bell-shaped. If in addition $f$ is smooth, then $f$ is bell-shaped. Conversely, any parameters $a$, $b$, $c$ and $\ph$ with the properties listed above \textup{(}where in~\ref{thm:bell:d} we assume that $\laplace f(i \xi)$ is defined by~\eqref{bell}\textup{)} correspond to some weakly bell-shaped function $f$ (possibly with an extra atom at $b$).
\end{theorem}

\begin{figure}
\centering
\includegraphics[width=0.8\textwidth]{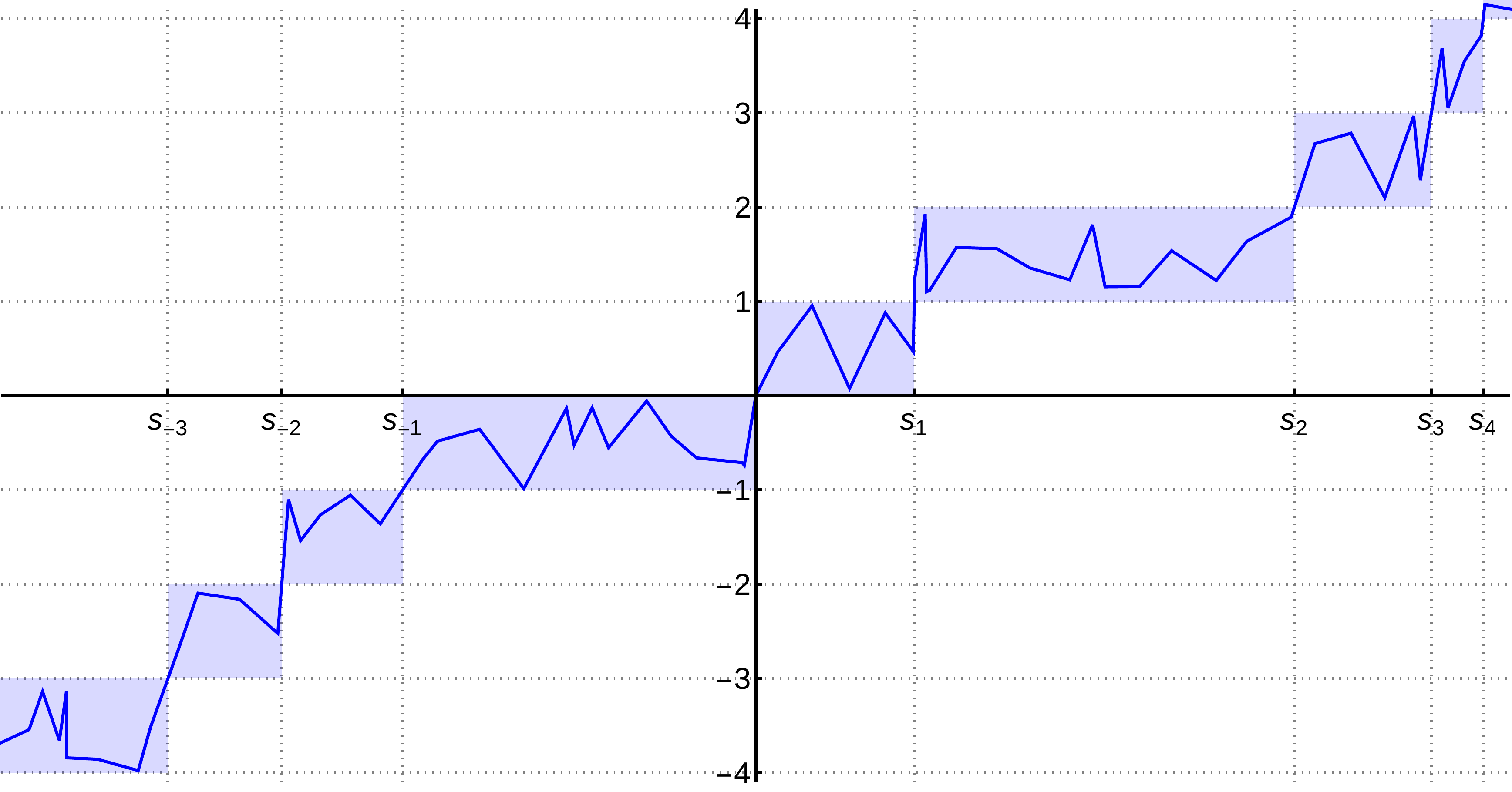}
\caption{Illustration for the level crossing condition~\ref{thm:bell:a} in Theorem~\ref{thm:bell}. The graph of the function $\ph$ is to be contained in the union of rectangles $[s_k, s_{k+1}] \times [k, k+1]$, $k \in \Z$.}
\label{fig:phi}
\end{figure}

The integrability condition~\ref{thm:bell:c} asserts convergence of the integral in the expression for $\laplace f$. Condition~\ref{thm:bell:d} is also rather natural: as it is explained in Section~\ref{sec:main}, it is related to convergence of the corresponding function $f$ to zero at $\pm\infty$. However, the \emph{level crossing condition}~\ref{thm:bell:a} may seem a bit artificial: we require that there is a non-decreasing sequence of points $s_k \in [-\infty, \infty]$, where $k \in \Z$, such that $\ph(s) \in [k, k + 1]$ for $s \in (s_k, s_{k+1})$ (see Figure~\ref{fig:phi}). The function $\ph$ is therefore uniformly close to a non-decreasing function, but it may have downward jumps of size up to $1$, as long as these jumps do not cross any integer. Interestingly, this rather unnatural assumption on $\ph$ cannot be easily relaxed, as shown by the examples in Section~\ref{sec:ex}.

The class of bell-shaped functions $f$ described by Theorem~\ref{thm:bell} includes both integrable and non-integrable functions. If we assume in addition that $f$ is integrable, then, up to multiplication by a constant, $f$ is the density function of a certain infinitely divisible distribution. In this case $a$ is its Gaussian coefficient, and the L\'evy measure has a very simple description in terms of the function $\ph$ defined in Theorem~\ref{thm:bell}: it is of the form $\nu(x) dx$, where for $x > 0$, $\nu(x)$ and $\nu(-x)$ are Laplace transforms of $\ph(s) \ind_{(0, \infty)}(s)$ and $-\ph(-s) \ind_{(0, \infty)}(s)$, respectively. Since $\ph(s)$ and $-\ph(-s)$ are non-negative for $s > 0$, $\nu(x)$ and $\nu(-x)$ are completely monotone functions on $(0, \infty)$, that is, $f$ is the density function of the one-dimensional distribution of a Lévy process with completely monotone jumps, studied in~\cite{bib:k19,bib:r83}. Note, however, that not all completely monotone densities $\nu(x)$ and $\nu(-x)$ can be obtained in this way.

The expression for the Fourier transform of $f$ in Theorem~\ref{thm:bell} can be written in a L\'evy--Khintchine fashion in the general case, also when $f$ is not integrable. This is shown in Corollary~\ref{cor:bell} below, which covers a slightly narrower class of functions; in the general statement, one would assume that $\nu(x)$ and $\nu(-x)$ are Laplace transforms of $\ph(s) \ind_{(0, \infty)}(s)$ and $-\ph(-s) \ind_{(0, \infty)}(s)$, respectively, where $\ph$ satisfies the conditions listed in Theorem~\ref{thm:bell}.

With $\ph$ and $\nu$ defined as in the previous paragraph, the function $\ph$ is non-decreasing if and only if $x \nu(x)$ and $x \nu(-x)$ are completely monotone functions of $x > 0$. This condition characterises a class of functions which is sometimes called \emph{extended generalised gamma convolutions}; we refer to Chapter~7 in~\cite{bib:b92} for a detailed discussion. Theorem~\ref{thm:bell} asserts in particular that all such functions are weakly bell-shaped.

\begin{corollary}
\label{cor:bell}
Suppose that $f$ is a locally integrable function which converges to zero at $\pm\infty$, which is decreasing near $\infty$ and increasing near $-\infty$. Suppose furthermore that for $\xi \in \R \setminus \{0\}$ the Fourier transform of $f$ satisfies
\formula{
 \laplace f(i \xi) & = \exp\biggl(-a \xi^2 - i b \xi + c + \int_{-\infty}^\infty (e^{-i \xi x} - (1 + |x| - i \xi x) e^{-|x|}) \nu(x) dx \biggr) ;
}
here $a \ge 0$, $b, c \in \R$ and $\nu : \R \to \R$ is a function with the following properties:
\begin{enumerate}[label=\textnormal{(\alph*)}]
\item\label{cor:bell:a} $x \nu(x)$ and $x \nu(-x)$ are completely monotone functions of $x > 0$\textup{;}
\item\label{cor:bell:b} we have
\formula{
 \int_{-1}^1 x^2 \nu(x) dx < \infty .
}
\end{enumerate}
Then $f$ is weakly bell-shaped. If in addition $f$ is smooth, then $f$ is bell-shaped. Conversely, any parameters $a$, $b$, $c$ and $\nu$ with the properties listed above correspond to some weakly bell-shaped function $f$ (possibly with an extra atom at $b$).

In particular, if $f$ is the density function of an infinitely divisible distribution with Gaussian coefficient $a \ge 0$, drift $b \in \R$ and L\'evy measure $\nu(x) dx$, and if $x \nu(x)$ and $x \nu(-x)$ are completely monotone on $(0, \infty)$, then $f$ is weakly bell-shaped.
\end{corollary}

As a special case, we obtain a result that was given in~\cite{bib:g84} with an erroneous proof, unless the index of stability is $2$ or $1/n$ for some $n = 1, 2, \ldots$; we refer to~\cite{bib:s15} for a detailed discussion.

\begin{corollary}
\label{cor:stable}
All stable distributions on $\R$ have bell-shaped density functions.
\end{corollary}

As mentioned above, Corollary~\ref{cor:stable} was proved by T.~Simon in~\cite{bib:s15} in the one-sided case, that is, when the distribution is concentrated on $(0, \infty)$. The argument used in~\cite{bib:s15} involves the representation of the density $f$ of a~positive stable distribution as a convolution of a P\'olya frequency function and a~completely monotone function on $(0, \infty)$. Although this is not stated explicitly in~\cite{bib:s15}, the proof given there can be easily adapted to prove a special case of Theorem~\ref{thm:bell} corresponding to functions $f$ which are concentrated on $(0, \infty)$; this extension is in fact applied in~\cite{bib:js15}. A remark at the end of~\cite{bib:s15} explains that a similar proof in the two-sided case is not possible, because one of the convolution factors is no longer completely monotone.

The idea of the proof of Theorem~\ref{thm:bell} is very much inspired by T.~Simon's work: we show that it is enough to represent $f$ as a convolution of a P\'olya frequency function and what we call an \emph{absolutely monotone-then-completely monotone} function. This, however, requires a completely different approach, which turns out to be shorter and more elementary than the method of~\cite{bib:s15}.

Three plausible extensions of Theorem~\ref{thm:bell} are disproved in Section~\ref{sec:ex}. A complete description of all bell-shaped functions remains a widely open problem: there is no good conjecture on their characterisation, and the author would be very surprised if Theorem~\ref{thm:bell} described all of them. This question appears to be closely related to the study of zeroes of holomorphic functions; we refer the reader to~\cite{bib:be06,bib:f08,bib:k68,bib:kk00,bib:p43} for further discussion and references.

The article is structured as follows. After a formal definition of the class of bell-shaped functions in Section~\ref{sec:bell}, we prove in Section~\ref{sec:amcm} that absolutely monotone-then-completely monotone functions are weakly bell-shaped. The definition of P\'olya frequency functions and their variation diminishing property is discussed in Section~\ref{sec:polya}. Section~\ref{sec:main} contains the proof of main results. Finally, in Section~\ref{sec:ex} we discuss some examples and counter-examples.

\subsection*{Acknowledgements}

I learned about the bell-shape from Thomas Simon during his seminar talk in Wroc{\l}aw, and I thank him for stimulating discussions about the problems considered here. I thank Alexandre Eremenko for sharing with me his broader view on the subject, and in particular for letting me know about references~\cite{bib:h50,bib:kk00,bib:p43,bib:hw55} in a \emph{MathOverflow} discussion at~\cite{bib:mo}. I also thank Takahiro Hasebe and the anonymous referee for pointing out errors in the preliminary version of this article.

%
%

\section{Bell-shaped functions}
\label{sec:bell}

All functions and measures in this article are Borel, and if $f$ is a measure, we denote the density function of (the absolutely continuous part of) $f$ by the same symbol $f$. A~function or a measure is said to be concentrated on a given set if it is equal to zero in the complement of this set. The Gauss--Weierstrass kernel is defined by
\formula{
 G_t(x) & = \frac{1}{\sqrt{4 \pi t}} \, e^{-x^2 / (4 t)}
}
for $t > 0$ and $x \in \R$. We say that a function $f : \R \to \R$ \emph{changes its sign} exactly $n$ times if $n + 1$ is the maximal length $m$ of a strictly increasing sequence $x_1, x_2, \ldots, x_m \in \R$ such that $f(x_j) f(x_{j+1}) < 0$ for $j = 1, 2, \ldots, m - 1$. If $f$ is differentiable and $f'(x) \ne 0$ whenever $f(x) = 0$, then $f$ changes its sign $n$ times if and only if it has $n$ zeroes.

\begin{definition}
A smooth function function $f : \R \to \R$ is said to be \emph{strictly bell-shaped} if for every $n = 0, 1, 2, \ldots$ the $n$-th derivative $f^{(n)}$ of $f$ converges to zero at $\pm \infty$ and $f^{(n)}$ changes its sign exactly $n$ times.

A function $f : \R \to \R$ (or a measure $f$ on $\R$) is said to be \emph{weakly bell-shaped} if for every $t > 0$ the function $f * G_t$ is well-defined and strictly bell-shaped.
\end{definition}

If a measure $f$ is weakly bell-shaped, then for every $t > 0$ the function $f * G_t$ is unimodal. Since $f * G_t$ converges vaguely to $f$ as $t \to 0^+$, $f$ is necessarily a unimodal measure: it may contain an atom at some $b \in \R$, and it has a unimodal density function on $\R \setminus \{b\}$ with a maximum at $b$.

In Section~\ref{sec:polya} we will see that strictly bell-shaped functions are weakly bell-shaped, and that for a function (or a measure) $f$ to be weakly bell-shaped it is sufficient to assume that $f * G_t$ is strictly bell-shaped for some sequence of $t > 0$ that converges to zero.

It is easy to see that a pointwise limit of a sequence of functions that change their sign exactly $n$ times is a function that changes its sign at most $n$ times. As a consequence, if $f_k$ is a sequence of strictly bell-shaped functions such that as $k \to \infty$ the derivatives $f_k^{(n)}$, $n = 0, 1, 2, \ldots$, converge pointwise to the corresponding derivatives $f^{(n)}$ of some function~$f$, and additionally $f$ tends to zero at $\pm \infty$, then either $f$ is constant zero or $f$ is strictly bell-shaped. In particular, if $f$ is a weakly bell-shaped smooth function, then $f$ is strictly bell-shaped: as $t \to 0^+$, the derivatives of strictly bell-shaped functions $f * G_t$ converge pointwise to the corresponding derivatives of $f$.

In a similar way, if $f_k$ is a sequence of weakly bell-shaped functions (or measures) which converges vaguely to a function (or a measure) $f$ as $k \to \infty$, and additionally $f$ tends to $0$ at $\pm\infty$, then either $f$ is constant zero or $f$ is weakly bell-shaped. Indeed: if $f$ is not identically zero and it converges to zero at $\pm \infty$, then the modes of $f_k$ are necessarily bounded as $k \to \infty$. This implies that the derivatives of $f_k * G_t$ converge pointwise to the derivatives of $f * G_t$, and so $f * G_t$ is either constant zero or strictly bell-shaped.

%
%

\section{$\amcm$ functions}
\label{sec:amcm}

In this section introduce a class of $\amcm$ functions and we prove that they are weakly bell-shaped. We begin with definitions and notation. We say that a smooth function $f : \R \to \R$ has a \emph{zero of multiplicity $N$} at $x$ if $f^{(n)}(x) = 0$ for $n = 0, 1, 2, \ldots, N - 1$ and $f^{(N)}(x) \ne 0$. We denote one-sided limits of a function $f$ at $x$ by $f(x^+)$ and $f(x^-)$. We will often consider measures that are sums of an absolutely continuous part (which we identify with the corresponding density function) and a Dirac measure. Since we are more tempted to think about these measures as functions with an extra atom, we will call such measures \emph{extended functions}.

\begin{definition}
A function $f : \R \to \R$ is \emph{completely monotone} on an open interval $I$ if it is smooth in $I$ and $(-1)^n f^{(n)}(x) \ge 0$ for every $n = 0, 1, 2, \ldots$ and $x \in I$. When the inequality $f^{(n)}(x) \ge 0$ is satisfied instead, then $f$ is said to be \emph{absolutely monotone} on $I$. We write $\cm_+$ for the class of completely monotone functions on $(0, \infty)$ and $\cm$ for the class of completely monotone functions on $\R$. Similarly, we write $\am_-$ and $\am$ for the classes of absolutely monotone functions on $(-\infty, 0)$ and on $\R$, respectively.

We say that $f : \R \to \R$ is \emph{absolutely monotone-then-completely monotone} if it is absolutely monotone on $(-\infty, 0)$ and completely monotone on $(0, \infty)$. More generally, we allow $f$ to be an extended function, comprising an absolutely monotone-then-completely monotone function on $\R \setminus \{0\}$, and possibly a non-negative atom at $0$. We write $\amcm$ for the class of absolutely monotone-then-completely monotone extended functions.
\end{definition}

The Laplace transform of a measure $\mu$ on $\R$ or a function $f : \R \to \R$ is defined by
\formula{
 \laplace \mu(x) & = \int_{\R} e^{-s x} \mu(ds) , &\qquad
 \laplace f(z) & = \int_{-\infty}^\infty e^{-x z} f(x) dx ,
}
whenever the integrals converge. By Bernstein's theorem, $f \in \cm_+$ if and only if $f(x) = \laplace \mu(x)$ for $x > 0$, where $\mu$ is some non-negative measure concentrated on $[0, \infty)$, whose Laplace transform is convergent on $(0, \infty)$. The measure $\mu$ is uniquely determined by $f$, and it is often called the \emph{Bernstein measure} of $f$. Clearly, $f$ converges to zero at $\infty$ if and only if $\mu(\{0\}) = 0$.

By Fubini's theorem, $f \in \cm_+$ is locally integrable if and only if the Bernstein measure $\mu$ of $f$ satisfies
\formula[eq:loc:int]{
 \int_{[0, \infty)} \frac{1}{1 + s} \, \mu(ds) < \infty .
}
In this case, again by Fubini's theorem, the Laplace transform of $f$ is the Stieltjes transform of $\mu$, that is,
\formula[eq:cm:stieltjes]{
 \laplace f(z) & = \int_0^\infty e^{-x z} f(x) dx = \int_{[0, \infty)} \frac{1}{z + s} \, \mu(ds)
}
when $\re z > 0$. The right-hand side is well-defined for all $z \in \C \setminus (-\infty, 0]$, and if $f$ is integrable, then~\eqref{eq:cm:stieltjes} holds also when $z \in i \R$. We will need the following extension of this property.

\begin{lemma}\label{lem:cm:reg}
Suppose that $f \in \cm_+$, $f$ is locally integrable and $f$ converges to zero at~$\infty$. Let $\mu$ be the Bernstein measure of $f$. Then
\formula[eq:cm:stieltjes:reg]{
 \laplace f(z) & = \int_0^\infty e^{-z x} f(x) dx = \int_{(0, \infty)} \frac{1}{z + s} \, \mu(ds)
}
for all $z \in i \R \setminus \{0\}$, where the former integral is understood as an improper integral.
\end{lemma}

\begin{proof}
Note that $\mu(\{0\}) = 0$ because $f$ converges to zero at $\infty$. Let $z \in i \R \setminus \{0\}$. By~\eqref{eq:cm:stieltjes}, for $\eps > 0$ we have
\formula{
 \laplace f(z + \eps) & = \int_{[0, \infty)} \frac{1}{z + \eps + s} \, \mu(ds) .
}
The dominated convergence theorem implies that the right-hand side converges to the right-hand side of~\eqref{eq:cm:stieltjes:reg} as $\eps \to 0^+$. On the other hand, the left-hand side can be written as
\formula{
 \laplace f(z + \eps) & = \int_0^1 e^{-x (z + \eps)} f(x) dx + \int_1^\infty e^{-x (z + \eps)} f(x) dx \\
 & = \int_0^1 e^{-x (z + \eps)} f(x) dx + \frac{e^{-z - \eps} f(1)}{z + \eps} + \int_1^\infty \frac{e^{-x (z + \eps)} f'(x)}{z + \eps} \, dx ;
}
in the second equality we integrated by parts. Again by the dominated convergence theorem,
\formula{
 \lim_{\eps \to 0^+} \laplace f(z + \eps) & = \int_0^1 e^{-x z} f(x) dx + \frac{e^{-z} f(1)}{z} + \int_1^\infty \frac{e^{-x z} f'(x)}{z} \, dx .
}
Another integration by parts leads to
\formula{
 \lim_{\eps \to 0^+} \laplace f(z + \eps) & = \int_0^1 e^{-x z} f(x) dx + \lim_{\alpha \to \infty} \biggl(\frac{e^{-\alpha z} f(\alpha)}{z} + \int_1^\alpha e^{-x z} f(x) dx\biggr) \\
 & = \lim_{\alpha \to \infty} \int_0^\alpha e^{-x z} f(x) dx ,
}
and the proof is complete.
\end{proof}

For a detailed treatment of completely monotone functions, Stieltjes functions and related notions, we refer the reader to~\cite{bib:ssv12}.

Denote $\check{f}(x) = f(-x)$. Clearly, $f \in \cm_+$ if and only if $\check{f} \in \am_-$, and $f \in \cm$ if and only if $\check{f} \in \am$. It follows that $f \in \amcm$ if and only if $f$ has a non-negative atom at $0$ of mass $m = f(\{0\})$, and there are non-negative measures $\mu_+$ and $\mu_-$ concentrated on $[0, \infty)$ such that
\formula{
 f(x) & = \laplace \mu_-(-x) &\qquad& \text{for $x < 0$,} \\
 f(x) & = \laplace \mu_+(x) &\qquad& \text{for $x > 0$.}
}
Clearly, the measures $\mu_+$ and $\mu_-$ are uniquely determined by $f$. By an analogy with the case of completely monotone functions, we call $\mu_+$ and $\mu_-$ the \emph{Bernstein measures} of the function $f \in \amcm$. In Section~\ref{sec:main} we will need the following result.

\begin{corollary}\label{cor:amcm:reg}
Suppose that $f \in \amcm$, $f$ is locally integrable and $f$ converges to zero at $\pm\infty$. Let $\mu_+$ and $\mu_-$ be the Bernstein measures of $f$. Then
\formula[eq:amcm:stieltjes:reg]{
 \laplace f(z) & = \int_{-\infty}^\infty e^{-z x} f(x) dx = f(\{0\}) + \int_{(0, \infty)} \frac{1}{z + s} \, \mu_+(ds) - \int_{(0, \infty)} \frac{1}{z - s} \, \mu_-(ds)
}
for all $z \in i \R \setminus \{0\}$, where the integral in the definition of $\laplace f(z)$ is understood as an improper integral.
\end{corollary}

\begin{proof}
It suffices to apply Lemma~\ref{lem:cm:reg} to completely monotone functions $f \ind_{(0, \infty)}$ and $\check{f} \ind_{(0, \infty)}$, and note that
\formula[]{\notag
 \lim_{\alpha \to -\infty} \int_{(\alpha, 0)} e^{-z x} f(x) dx & = \lim_{\alpha \to \infty} \int_{(0, \alpha)} e^{z x} \check{f}(x) dx = \int_{(0, \infty)} \frac{1}{-z + s} \, \mu_-(ds) . \qedhere
}
\end{proof}

The main result of this section, Theorem~\ref{thm:amcm}, requires three auxiliary statements.

\begin{lemma}
\label{lem:cm:rec}
Let $f(x) = \laplace \mu(x)$ for $x > 0$ and $f(x) = 0$ for $x \le 0$, where $\mu$ is a measure concentrated on $[0, \infty)$ with all moments finite. Then for every $t > 0$ and $n = 0, 1, 2, \ldots$ there is a polynomial $P$ of degree at most $n - 1$ and a function $u \in \am$ such that
\formula[eq:cm:rec]{
 (f * G_t)^{(n)}(x) & = (P(x) + (-1)^n u(x)) G_t(x)
}
for all $x \in \R$. (For $n = 0$ we understand that $P$ is constant zero).
\end{lemma}

\begin{proof}
We proceed by induction. For $n = 0$ the function
\formula{
 u(x) & = \frac{1}{G_t(x)} \int_0^\infty f(y) G_t(x - y) dy \\
 & = \int_0^\infty f(y) e^{-(y^2 - 2 x y) / (4 t)} dy \\
 & = \int_0^\infty e^{x y / (2 t)} f(y) e^{-y^2 / (4 t)} dy \\
 & = 2 t \int_0^\infty e^{s x} f(2 t s) e^{-t s^2} ds
}
is clearly absolutely monotone on $\R$, as desired.

Suppose now that property~\eqref{eq:cm:rec} holds for some fixed $n$, and apply it to the function~$g$ defined by $g(x) = -f'(x)$ for $x > 0$, $g(x) = 0$ for $x \le 0$. Note that on $(0, \infty)$, $g$ is the Laplace transform of the measure $s \mu(ds)$, and $s \mu(ds)$ also has all moments finite, so $g$ satisfies the assumptions for~\eqref{eq:cm:rec}. It follows that
\formula{
 (g * G_t)^{(n)} & = (P + (-1)^n u) G_t
}
for some polynomial $P$ of degree $n - 1$ and some $u \in \am$. Recall that $\mu$ is a finite measure, so that $f(0^+) = \mu(\R)$ is finite (here the assumption on the moments of $\mu$ is used). It follows that
\formula{
 (f * G_t)'(x) & = f(0^+) G_t(x) - g * G_t(x)
}
for all $x \in \R$. Therefore,
\formula{
 (f * G_t)^{(n+1)} & = f(0^+) G_t^{(n)}(x) - (g * G_t)^{(n)}(x) \\
 & = \biggl(\frac{f(0^+) G_t^{(n)}(x)}{G_t(x)} - P(x) - (-1)^n u(x)\biggr) G_t(x)
}
for all $x \in \R$. Since $G_t^{(n)} / G_t$ is a polynomial of degree $n$ (namely, this is the $n$-th Hermite polynomial evaluated at $x / \sqrt{2 t}$, up to multiplication by a constant), the desired result for the derivative of degree $n + 1$ follows. This completes the proof by induction.
\end{proof}

By replacing $f$ with $\check{f}$, we immediately obtain the following corollary.

\begin{corollary}
\label{cor:am:rec}
Let $f(x) = \laplace \mu(-x)$ for $x < 0$ and $f(x) = 0$ for $x \ge 0$, where $\mu$~is a~measure concentrated on $[0, \infty)$ with all moments finite. Then for every $t > 0$ and $n = 0, 1, 2, \ldots$ there is a polynomial $P$ of degree at most $n - 1$ and a function $v \in \cm$ such that
\formula{
 (f * G_t)^{(n)}(x) & = (P(x) + v(x)) G_t(x)
}
for all $x \in \R$. (For $n = 0$ we understand that $P$ is constant zero).\qed
\end{corollary}

In order to prove the main theorem of this section, we need one more technical result.

\begin{lemma}
\label{lem:poly:rec}
Suppose that $u \in \am$, $v \in \cm$, and $P$ is a polynomial of degree at most $n$, with coefficient at $x^n$ non-negative. Then the equation
\formula[eq:poly:rec]{
 P(x) + u(x) + (-1)^n v(x) & = 0 ,
}
if not satisfied for all $x \in \R$, has at most $n$ real solutions (counting multiplicities).
\end{lemma}

\begin{proof}
Again, we proceed by induction. For $n = 0$, the left-hand side of~\eqref{eq:poly:rec} is either constant zero or everywhere strictly positive, as desired. Suppose now that the assertion of the lemma is true for some $n$. Consider a function $f = P + u + (-1)^{n+1} v$, where $u \in \am$, $v \in \cm$, and $P$ is a polynomial of degree at most $n + 1$, with coefficient at $x^{n+1}$ non-negative. Suppose that $f$ has more than $n + 1$ real zeroes (counting multiplicities). By Rolle's theorem and the definition of the multiplicity of a zero, the derivative of $f$ has more than $n$ zeroes (counting multiplicities). However,
\formula{
 f' & = P' + u' + (-1)^n (-v') ,
}
and in the right-hand side $u' \in \am$, $-v' \in \cm$ and $P'$ is a polynomial of degree at most~$n$, with coefficient at $x^n$ non-negative. By the induction hypothesis, $f'$ either has at most $n$ zeroes (counting multiplicities) or it is constant zero. Since we already know that $f'$ has at least $n + 1$ zeroes, it follows that $f'$ is identically zero, which means that $f$ is constant. Since $f$ has at least one zero, $f$ is constant zero, as desired. This completes the proof by induction.
\end{proof}

Recall that an extended function $f \in \amcm$ is locally integrable if and only if the corresponding Bernstein measures $\mu_+$ and $\mu_-$ satisfy condition~\eqref{eq:loc:int}.

\begin{theorem}
\label{thm:amcm}
If $f \in \amcm$, $f$ is locally integrable, $f$ is not identically zero, but $f$ converges to zero at $\pm\infty$, then $f$ is weakly bell-shaped.
\end{theorem}

\begin{proof}
Recall that $f$ may have a non-negative atom at zero, and denote $m = f(\{0\}) \ge 0$. On $\R \setminus \{0\}$, $f$ is a pure function, completely monotone on $(0, \infty)$ and absolutely monotone on $(-\infty, 0)$. We first assume that the Bernstein measures $\mu_+$, $\mu_-$ of $f$ have all moments finite. In this case, by Lemma~\ref{lem:cm:rec} and Corollary~\ref{cor:am:rec}, for every $t > 0$ and $n = 0, 1, 2, \ldots$ there are functions $u \in \am$ and $v \in \cm$, as well as a polynomial $P$ of~degree at most $n - 1$, such that
\formula{
 (f * G_t)^{(n)}(x) & = (-1)^n \biggl(m \, \frac{(-1)^n G_t^{(n)}(x)}{G_t(x)} + P(x) + u(x) + (-1)^n v(x)\biggr) G_t(x)
}
for all $x \in \R$. (For $n = 0$ we understand that $P$ is constant zero). However, $(-1)^n G_t^{(n)} / G_t$ is a polynomial of degree $n$, with coefficient at $x^n$ positive. Therefore, by Lemma~\ref{lem:poly:rec}, the equation $(f * G_t)^{(n)}(x) = 0$ has at most $n$ solutions (counting multiplicities). As a consequence, $f * G_t$ is strictly bell-shaped. Since $t > 0$ is arbitrary, $f$ is weakly bell-shaped, as desired.

For general Bernstein measures $\mu_+$, $\mu_-$ we proceed by approximation. For $k = 1, 2, \ldots$ we define the extended function $f_k \in \amcm$ so that $f_k(\{0\}) = f(\{0\})$ and on $\R \setminus \{0\}$,
\formula{
 f_k(x) & = \int_{[-k, 0]} e^{s x} \mu_+(d x) &\qquad& \text{for $x < 0$,} \\
 f_k(x) & = \int_{[0, k]} e^{-s x} \mu_+(d x) &\qquad& \text{for $x > 0$.}
}
By the first part of the proof, for every $k = 1, 2, \ldots$ the extended function $f_k$ is weakly bell-shaped, unless it is identically zero. Furthermore, the density functions of $f_k$ converge monotonically to the density function of $f$ as $k \to \infty$, and so $f$ is either identically zero or weakly bell-shaped.
\end{proof}

\begin{example}
For any $p \in (0, 1)$ the functions $f(x) = |x|^{-p}$ and $g(x) = x^{-p} \ind_{(0, \infty)}(x)$ are weakly bell-shaped.
\end{example}

\begin{remark}
\label{rem:amcm:broad}
Let us say that an extended function $f$ is \emph{strictly bell-shaped in the broad sense} if we only require that $f^{(n)}$ converges to zero at $\pm \infty$ and changes its sign exactly $n$ times for $n = 1, 2, \ldots$, but not necessarily for $n = 0$. Similarly, we introduce the notion of a \emph{weakly bell-shaped function in the broad sense}.

By repeating the proof of Theorem~\ref{thm:amcm} it is easy to see that a locally integrable extended function $f$ is weakly bell-shaped in the broad sense if we assume that $f(\{0\}) \ge 0$, $f$ is locally integrable, $-f'$ is completely monotone on $(0, \infty)$ and $f'$ is absolutely monotone on $(-\infty, 0)$. In this case we have
\formula{
 f(x) & = c_- + \int_{(0, \infty)} (e^{s x} - e^{-s}) \mu_-(ds) &\qquad& \text{for $x < 0$,} \\
 f(x) & = c_+ + \int_{(0, \infty)} (e^{-s x} - e^{-s}) \mu_+(ds) &\qquad& \text{for $x > 0$,}
}
where $c_- = f(-1)$ and $c_+ = f(1)$ are arbitrary real constants, and $\mu_-$ and $\mu_+$ are arbitrary non-negative measures concentrated on $(0, \infty)$ such that
\formula{
 \int_{(0, \infty)} \frac{s}{(1 + s)^2} \, \mu_+(ds) & < \infty, & \int_{(0, \infty)} \frac{s}{(1 + s)^2} \, \mu_-(ds) & < \infty .
}
Such a function $f$ need not be positive, and it may converge to $-\infty$ at~$\infty$ or at~$-\infty$. The derivative $f'$ of $f$ is, however, completely monotone on $(0, \infty)$, and $-f'$ is absolutely monotone on $(-\infty, 0)$.
\end{remark}

\begin{example}
For any $p \in (0, 1)$ the functions $f(x) = -|x|^p$ and $g(x) = -x^p \ind_{(0, \infty)}(x)$ are weakly bell-shaped in the broad sense. As can be directly checked, for any $p \in (0, 1/2)$ the function $h(x) = -(1 + x^2)^p$ is strictly bell-shaped in the broad sense. 
\end{example}

%
%

\section{P\'olya frequency functions}
\label{sec:polya}

The class of P\'olya frequency functions is the closure (with respect to vague convergence of measures) of the class of convolutions of exponential distributions. The best way to describe this class involves the Fourier transform (or the characteristic function), which we identify with the restriction of the Laplace transform to the imaginary axis $i \R$. For this reason we re-use the notation for the Laplace transform and denote the Fourier transform of $f$ by $\laplace f(z)$, where we understand that $z \in i \R$.

\begin{definition}
\label{def:polya}
An integrable function $f : \R \to \R$ is a \emph{P\'olya frequency function} if its Fourier transform satisfies
\formula{
 \laplace f(z) & = e^{a z^2 - b z} \prod_{n = 1}^N \frac{e^{z / z_n}}{1 + z / z_n}
}
for all $z \in i \R$; here $a \ge 0$, $b \in \R$, $N \in \{0, 1, 2, \ldots, \infty\}$, $z_n \in \R \setminus \{0\}$, and
\formula{
 \sum_{n = 1}^N \frac{1}{|z_n|^2} < \infty .
}
For simplicity, we abuse the notation and agree that Dirac measures $\delta_b$ (which correspond to $a = 0$ and $N = 0$) are also P\'olya frequency functions, so formally a P\'olya frequency function is an extended function.
\end{definition}

The definition can be equivalently phrased as follows: $f$ is a P\'olya frequency function if and only if $f$ is the convolution of the Gauss--Weierstrass kernel $G_a$ (if $a > 0$), the Dirac measure $\delta_b$ and the (finite or infinite) family of shifted exponential measures
\formula{
 & |z_n| e^{-z_n (x - 1)} \ind_{(0, \infty)}(z_n (x - 1)) dx
}
with mean $0$ and variance $|z_n|^{-2}$.

\begin{definition}
An integrable extended function $f$ is said to be a \emph{variation diminishing convolution kernel} if for all bounded functions $g$ the function $f * g$ changes its sign at most as many times as the function $g$ does.
\end{definition}

We recall the following fundamental result of Schoenberg, proved originally in~\cite{bib:s47,bib:s48}.

\begin{theorem}[Schoenberg; see Chapter~IV in~\cite{bib:hw55}]
\label{thm:polya}
An integrable function is a variation diminishing convolution kernel if and only if it is a P\'olya frequency function, up to multiplication by a constant.\qed
\end{theorem}

We remark that we will only need the direct part of the above theorem, that is, the fact that P\'olya frequency functions are variation diminishing convolution kernels. The proof of this statement is relatively simple, and we sketch it for completeness. If $f(x) = e^{-x} \ind_{(0, \infty)}(x)$, then $e^x (f * g)(x) = \int_{-\infty}^x e^y f(y) dy$, which implies that $f * g$ changes its sign at most as many times as $g$. Therefore, $f$ is a variation diminishing convolution kernel. It follows that $f_{a,b}(x) = a f(a x + b)$ is a variation diminishing convolution kernel whenever $a, b \in \R$, $a \ne 0$. The class of non-negative variation diminishing convolution kernels with $L^1(\R)$ norm equal to $1$ is clearly closed in $L^1$ and closed under convolutions. It remains to observe that every P\'olya frequency function can be approximated in $L^1(\R)$ by convolutions of~$f_{a,b}$.

On the other hand, the converse part of Theorem~\ref{thm:polya} is a deep result that involves the concept of total positivity.

Since the Gauss--Weierstrass kernel is strictly bell-shaped, Schoenberg's theorem asserts that P\'olya frequency functions are weakly bell-shaped. It also implies the following properties of bell-shaped functions, which were already mentioned in Section~\ref{sec:bell}.

\begin{corollary}
If $f$ is a strictly bell-shaped function, then it is also a weakly bell-shaped function. A non-negative extended function $f$ is weakly bell-shaped if and only if $f * G_t$ is well-defined for every $t > 0$ and strictly bell-shaped for some sequence of $t > 0$ that converges to $0$.
\end{corollary}

\begin{proof}
For every $t > 0$, the Gauss--Weierstrass kernel $G_t$ is a P\'olya frequency function, and therefore it is a variation diminishing convolution kernel. Therefore, if $f$ is strictly bell-shaped, so is $f * G_t$ for every $t > 0$. Similarly, if $f$ is a non-negative extended function such that $f * G_t$ is strictly bell-shaped for some $t > 0$ and $f * G_s$ is well-defined for some $s > t$, then $f * G_s = (f * G_t) * G_{s - t}$ is strictly bell-shaped.
\end{proof}

%
%

\section{Synthesis}
\label{sec:main}

Clearly, the convolution of a bell-shaped function with a variation diminishing convolution kernel is again a bell-shaped function. In this section we describe the class of convolutions of $\amcm$ functions (which we already know to be bell-shaped) and P\'olya frequency functions (which are precisely variation diminishing convolution kernels) in terms of the Fourier transform. Recall that we identify the Fourier transform of $f$ with the restriction of the Laplace transform $\laplace f$ to the imaginary axis $i \R$.

In most results, we do not assume that $f$ is an integrable function. Instead, we suppose that $f$ is a locally integrable extended function, which converges to zero at $\pm \infty$ and which is monotone near $-\infty$ and near $\infty$. In this case the Fourier transform $\laplace f(z)$ is well-defined as an improper integral for $z \in i \R \setminus \{0\}$, and it uniquely determines the function $f$ among the class of extended functions just described.

By Corollary~\ref{cor:amcm:reg}, if $f \in \amcm$, $f$ is locally integrable and $f$ converges to zero at $\pm \infty$, then
\formula{
 \laplace f(z) & = m + \int_{(0, \infty)} \frac{1}{z + s} \, \mu_+(ds) - \int_{(0, \infty)} \frac{1}{z - s} \, \mu_-(ds)
}
for all $z \in i \R \setminus \{0\}$, where $m = f(\{0\})$, $\mu_+$ and $\mu_-$ are Bernstein measures of $f$, and $\mu_+, \mu_-$ satisfy the integrability condition~\eqref{eq:loc:int}. Furthermore, $f$ with the above properties is uniquely determined by the values of $\laplace f(z)$ for $z \in i \R \setminus \{0\}$. The next result provides an alternative expression for the Fourier transform of $f$. It is a variant of the equivalence of Stieltjes and exponential representations of Nevanlinna--Pick functions, studied in detail in~\cite{bib:ad56}. For completeness, we provide a proof.

\begin{proposition}
\label{prop:amcm:laplace}
The following conditions are equivalent:
\begin{enumerate}[label=\textnormal{(\alph*)}]
\item
\label{prop:amcm:laplace:a}
$F$ is the Laplace transform of a locally integrable $\amcm$ extended function which converges to zero at $\pm\infty$\textup{;} that is,
\formula[eq:amcm:stieltjes]{
 F(z) & = m + \int_{(0, \infty)} \frac{1}{z + s} \, \mu_+(ds) - \int_{(0, \infty)} \frac{1}{z - s} \, \mu_-(ds) ,
}
for all $z \in i \R \setminus \{0\}$, where $m \ge 0$ and $\mu_+$ and $\mu_-$ are non-negative measures satisfying the integrability condition
\formula[eq:amcm:int]{
 \int_{(0, \infty)} \frac{1}{1 + s} \, \mu_+(ds) & < \infty &\qquad& \text{and} &\qquad \int_{(0, \infty)} \frac{1}{1 + s} \, \mu_-(ds) & < \infty ;
}
\item
\label{prop:amcm:laplace:b}
either $F$ is constant zero on $i \R \setminus \{0\}$ or for all $z \in i \R \setminus \{0\}$ we have
\formula[eq:amcm:exp]{
 F(z) & = \exp\biggl(c + \int_{-\infty}^\infty \biggl( \frac{1}{z + s} - \frac{1}{s} \ind_{\R \setminus (-1, 1)}(s) \biggr) \ph(s) ds \biggr) ,
}
where $c \in \R$, $\ph$ takes values in $[0, 1]$ on $(0, \infty)$ and in $[-1, 0]$ on $(-\infty, 0)$, and
\formula[eq:amcm:extra]{
 \int_{-1}^1 \re F(i t) dt & < \infty &\qquad& \text{and} &\qquad \lim_{t \to 0^+} (t \im F(i t)) & = 0.
}
\end{enumerate}
\end{proposition}

\begin{proof}
Both expressions~\eqref{eq:amcm:stieltjes} and~\eqref{eq:amcm:exp} define a holomorphic function $F$ in $\C \setminus \R$. Indeed, the integrals in~\eqref{eq:amcm:stieltjes} and~\eqref{eq:amcm:exp} are absolutely convergent when $z \in \C \setminus \R$ by the assumptions on $\mu$ and $\ph$, and by a standard application of Morera's theorem, they define holomorphic functions. Furthermore, both expressions define a function $F$ satisfying $F(\overline{z}) = \overline{F(z)}$, so we may restrict our attention to the upper complex half-plane $\im z > 0$.

Formula~\eqref{eq:amcm:stieltjes} can be rewritten as
\formula[eq:amcmz:stieltjes]{
\begin{aligned}
 z F(z) & = m z + \int_{(0, \infty)} \biggl(1 - \frac{s}{z + s}\biggr) \mu_+(ds) - \int_{(0, \infty)} \biggl(1 + \frac{s}{z - s}\biggr) \mu_-(ds) \\
 & = m z + \int_{(0, \infty)} \biggl(\frac{-1}{z + s} + \frac{1}{s}\biggr) s \mu_+(ds) + \int_{(0, \infty)} \biggl(\frac{-1}{z - s} + \frac{1}{s} \biggr) s \mu_-(ds)
\end{aligned}
}
when $\im z > 0$. Similarly, using the identity
\formula{
 \log z & = \int_0^\infty \biggl(\frac{-1}{z + s} - \frac{1}{s} \ind_{(1, \infty)}(s)\biggr) ds ,
}
valid for $z \in \C \setminus (-\infty, 0]$, we find that~\eqref{eq:amcm:exp} is equivalent to
\formula[eq:amcmz:exp]{
 z F(z) & = \exp \biggl(c + \int_{-\infty}^\infty \biggl(\frac{-1}{z + s} + \frac{1}{s} \ind_{\R \setminus (-1, 1)}(s)\biggr)(\ind_{(0, \infty)}(s) - \ph(s)) ds\biggr)
}
when $\im z > 0$. Recall that $\pi^{-1} \im (-1 / (z - s)) = \pi^{-1} |z - s|^{-2} \im z$ is the Poisson kernel for the upper complex half-plane.

We first prove that condition~\ref{prop:amcm:laplace:a} implies~\ref{prop:amcm:laplace:b}. Suppose that $F$ is not constant zero, that is, $m > 0$ or at least one of the measures $\mu_+, \mu_-$ is non-zero. By~\eqref{eq:amcmz:stieltjes},
\formula[eq:amcmzim]{
 \im (z F(z)) & = m \im z + \int_{(0, \infty)} \frac{s \im z}{|z + s|^2} \, \mu_+(ds) + \int_{(0, \infty)} \frac{s \im z}{|z - s|^2} \, \mu_-(ds) > 0 ,
}
so that $z F(z)$ is a non-zero Nevanlinna--Pick function. It follows that $\arg(z F(z)) \in (0, \pi)$ when $\im z > 0$. Thus, $\arg(z F(z))$ is a bounded harmonic function in the upper complex half-plane. By the Poisson representation theorem, $\arg(z F(z))$ is the Poisson integral of the corresponding boundary values: if we denote
\formula{
 \Phi(s) & = \lim_{t \to 0^+} \arg ((s + i t) F(s + i t)) ,
}
then $\Phi(s)$ is well-defined for almost all $s \in \R$, $\Phi(s) \in [0, \pi]$ and
\formula{
 \arg (z F(z)) & = \frac{1}{\pi} \int_{-\infty}^\infty \im \frac{-1}{z - s} \, \Phi(s) ds = \frac{1}{\pi} \int_{-\infty}^\infty \im \frac{-1}{z + s} \, \Phi(-s) ds
}
when $\im z > 0$. Since $\arg F(z) = \im \log F(z)$, we have
\formula{
 \im \log (z F(z)) & = \frac{1}{\pi} \, \im \biggl(\int_{-\infty}^\infty \biggl( \frac{-1}{z + s} + \frac{1}{s} \ind_{\R \setminus (-1, 1)}(s) \biggr) \Phi(-s) ds\biggr) ;
}
the term $s^{-1} \ind_{\R \setminus (-1, 1)}(s)$ has zero imaginary part, and it is needed to make the integral convergent. Since two holomorphic functions with equal imaginary parts differ by a real constant, formula~\eqref{eq:amcmz:exp} follows, with $\ph(s) = \ind_{(0, \infty)}(s) - \pi^{-1} \Phi(-s)$. Clearly, $\ph(s) \in [0, 1]$ for $s > 0$ and $\ph(s) \in [-1, 0]$ for $s < 0$, as desired. By~\eqref{eq:amcm:stieltjes} (or~\eqref{eq:amcmzim}), we have
\formula{
 \int_0^1 \re F(i t) dt & = \int_0^1 \biggl( m + \int_{(0, \infty)} \frac{s}{t^2 + s^2} \, (\mu_+ + \mu_-)(ds) \biggr) dt .
}
Applying Fubini's theorem and using the inequality $t^2 + s^2 \ge \tfrac{1}{2} (t + s)^2$, we find that
\formula{
 \int_0^1 \re F(i t) dt & \le m + \int_{(0, \infty)} \biggl(\int_0^1 \frac{2 s}{(t + s)^2} \, dt\biggr) (\mu_+ + \mu_-)(ds) \\
 & = m + \int_{(0, \infty)} \frac{2}{1 + s} \, (\mu_+ + \mu_-)(ds) .
}
Since $\mu_+$ and $\mu_-$ satisfy the integrability condition~\eqref{eq:amcm:int}, the right-hand side is finite, and the first part~\eqref{eq:amcm:extra} follows. Finally, also the second part of~\eqref{eq:amcm:extra} is a consequence of~\eqref{eq:amcm:stieltjes}: we have
\formula[eq:amcm:lim]{
 \lim_{t \to 0^+} (t \im F(i t)) & = \lim_{t \to 0^+} \int_{(0, \infty)} \frac{t^2}{t^2 + s^2} \, \mu_+(ds) - \lim_{t \to 0^+} \int_{(0, \infty)} \frac{t^2}{t^2 + s^2} \, \mu_-(ds) = 0
}
by the dominated convergence theorem.

We now argue that condition~\ref{prop:amcm:laplace:b} implies~\ref{prop:amcm:laplace:a}. By~\eqref{eq:amcmz:exp}, we have
\formula{
 \arg(z F(z)) & = \int_{-\infty}^\infty \im \frac{-1}{z + s} \, (\ind_{(0, \infty)}(s) - \ph(s)) ds
}
when $\im z > 0$; that is, $\arg(z F(z))$ is the Poisson integral of $\pi (\ind_{(0, \infty)}(s) - \ph(s)) \in [0, \pi]$. It follows that $\arg(z F(z)) \in [0, \pi]$ when $\im z > 0$. Equivalently, $\im(z F(z)) \ge 0$, and hence $z F(z)$ is again a Nevanlinna--Pick function.

By Herglotz's theorem, the non-negative harmonic function $\im(z F(z))$ in the upper complex half-plane is the Poisson integral of a non-negative measure, that is, when $\im z > 0$, we have
\formula{
 \im(z F(z)) & = m \im z + \frac{1}{\pi} \int_{\R} \im \frac{-1}{z - s} \, \sigma(ds)
}
for some $m \ge 0$ and some non-negative measure $\sigma$ on $\R$ such that $\int_{\R} (1 + s)^{-2} \sigma(ds)$ is finite. Furthermore, by Fubini's theorem,
\formula{
 \int_0^1 \re F(i t) dt & = \int_0^1 \frac{\im(i t F(i t))}{t} \, dt \\
 & = \int_0^1 \biggl(m + \frac{1}{\pi t} \int_{\R} \im \frac{-1}{i t - s} \, \sigma(ds) \biggr) dt \\
 & = m + \frac{1}{\pi} \int_{\R} \int_0^1 \frac{1}{t^2 + s^2} \, dt \sigma(ds) \\
 & \ge \frac{1}{\pi} \int_{\R} \int_0^1 \frac{1}{(t + |s|)^2} \, dt \sigma(ds) \\
 & = \frac{1}{\pi} \int_{\R} \frac{1}{|s| (1 + |s|)} \, \sigma(ds) ,
}
where we understand that the integrand in the right-hand side is infinite at $s = 0$. The left-hand side is finite by the first part of~\eqref{eq:amcm:extra}. Therefore, $\sigma(\{0\}) = 0$, and if we define $\mu_+(ds) = \pi^{-1} s^{-1} \ind_{(0, \infty)}(s) \check{\sigma}(ds)$ and $\mu_-(ds) = \pi^{-1} s^{-1} \ind_{(0, \infty)}(s) \sigma(ds)$ (where $\check{\sigma}(A) = \sigma(-A)$), then $\mu_+$ and $\mu_-$ satisfy the integrability condition~\eqref{eq:amcm:int}. Furthermore,
\formula{
 \im(z F(z)) & = m \im z + \int_{(0, \infty)} \im \frac{-1}{z + s} \, s \mu_+(ds) + \int_{(0, \infty)} \im \frac{-1}{z - s} \, s \mu_-(ds) \\
 & = \im \biggl(m z + \int_{(0, \infty)} \biggl(\frac{-1}{z + s} + \frac{1}{s}\biggr) s \mu_+(ds) + \int_{(0, \infty)} \biggl(\frac{-1}{z - s} - \frac{1}{s}\biggr) s \mu_-(ds)\biggr) ;
}
the term $1 / s$ has zero imaginary part, and it makes the integrals convergent. Since two holomorphic functions with equal imaginary parts differ by a real constant, formula~\eqref{eq:amcmz:stieltjes} holds up to addition by a real constant $b$. It follows that~\eqref{eq:amcm:stieltjes} holds up to addition by $b / z$, that is,
\formula{
 F(z) & = \frac{b}{z} + m + \int_{(0, \infty)} \frac{1}{z + s} \, \mu_+(ds) - \int_{(0, \infty)} \frac{1}{z - s} \, \mu_-(ds) .
}
Finally, as in~\eqref{eq:amcm:lim}, we find that
\formula{
 \lim_{t \to 0^+} (t \im F(i t)) & = b ,
}
and so $b = 0$ by the second part of~\eqref{eq:amcm:extra}. We conclude that~\eqref{eq:amcm:stieltjes} holds, and the proof is complete.
\end{proof}

\begin{remark}
If $z F(z)$ is a Nevanlinna--Pick function, then $1 / F(z)$ is the characteristic (Laplace) exponent of a Lévy process $X_t$ with completely monotone jumps; we refer to~\cite{bib:k19,bib:r83} for details. Furthermore, $X_t$ is transient if and only if the first part of~\eqref{eq:amcm:extra} holds, and in this case $F(z)$ is the Laplace transform of the potential kernel of $X_t$; see, for example, \cite{bib:s99}. Finally, the second part of~\eqref{eq:amcm:extra} asserts that the potential kernel converges to zero at $\pm \infty$.
\end{remark}

The next result is merely a reformulation of the definition of a P\'olya frequency function.

\begin{proposition}
\label{prop:polya:laplace}
The following conditions are equivalent:
\begin{enumerate}[label=\textnormal{(\alph*)}]
\item $F(z)$ is the Laplace transform of a P\'olya frequency function, that is
\formula[eq:polya]{
 F(z) & = e^{a z^2 - b z} \prod_{n = 1}^N \frac{e^{z / z_n}}{1 + z / z_n}
}
for all $z \in i \R$, where, as in Definition~\ref{def:polya}, $a \ge 0$, $b \in \R$ and
\formula{
 \sum_{n = 1}^N \frac{1}{|z_n|^2} < \infty;
}
\item for all $z \in i \R$ we have
\formula[eq:polya:exp]{
 F(z) & = \exp\biggl(a z^2 - b z + \int_{-\infty}^\infty \biggl( \frac{1}{z + s} - \frac{1}{s} + \frac{z}{s^2} \biggr) \ph(s) ds \biggr) ,
}
where $a \ge 0$, $b \in \R$ and $\ph : \R \to \Z$ is a non-decreasing integer-valued function such that
\formula{
 \int_{-\infty}^\infty \frac{|\ph(s)|}{|s|^3} \, ds < \infty .
}
\end{enumerate}
\end{proposition}

\begin{proof}
Suppose that $z \in i \R$. Observe that if $z_n > 0$, then
\formula{
 z/z_n - \log(1 + z / z_n) & = \int_{z_n}^\infty \biggl(\frac{1}{s + z} - \frac{1}{s} + \frac{z}{s^2}\biggr) \, ds ,
}
while if $z_n < 0$, then
\formula{
 z/z_n - \log(1 + z / z_n) & = -\int_{-\infty}^{z_n} \biggl(\frac{1}{s + z} - \frac{1}{s} + \frac{z}{s^2}\biggr) \, ds .
}
Therefore, if $f$ is a P\'olya frequency function such that $\laplace f$ has the representation given in the statement of the theorem, and if
\formula[eq:polya:phi]{
 \ph(s) & = \sum_{n = 1}^N \bigl(\ind_{(0, \infty)}(z_n) \ind_{[z_n, \infty)}(s) - \ind_{(-\infty, 0)}(z_n) \ind_{(-\infty, z_n)}(s)\bigr) ,
}
then $\ph : \R \to \Z$ is non-decreasing and integer-valued, $\ph(0) = 0$,
\formula[eq:polya:phi:zn]{
 \int_{-\infty}^\infty \frac{\ph(s)}{s^3} \, ds & = \frac{1}{2} \sum_{n = 1}^\infty \frac{1}{|z_n|^2} \, ,
}
and
\formula{
 \prod_{n = 1}^N \, \frac{e^{z / z_n}}{1 + z / z_n} & = \exp\biggl(\int_{-\infty}^\infty \biggl( \frac{1}{z + s} - \frac{1}{s} + \frac{z}{s^2} \biggr) \ph(s) ds \biggr) ,
}
as desired. Conversely, if $\ph : \R \to \Z$ is a right-continuous non-decreasing integer-valued function and $|s^{-3} \ph(s)|$ is integrable, then necessarily $\ph(s) = 0$ for $s$ in some neighbourhood of $0$. It follows that $\ph$ can be written in the form given in~\eqref{eq:polya:phi}, and~\eqref{eq:polya:phi:zn} remains valid. Therefore, there is a corresponding P\'olya frequency function $f$ such that~\eqref{eq:polya:exp} is satisfied.
\end{proof}

We are now in position to describe the class of convolutions of $\amcm$ functions and P\'olya frequency functions.

\begin{lemma}
\label{lem:bell}
If $f$ is the convolution of a locally integrable $g \in \amcm$ which converges to zero at $\pm\infty$ and a P\'olya frequency function $h$, then
\formula[eq:bell]{
 \laplace f(z) & = \exp\biggl(a z^2 - b z + c + \int_{-\infty}^\infty \biggl( \frac{1}{z + s} - \biggl(\frac{1}{s} - \frac{z}{s^2} \biggr) \ind_{\R \setminus (-1, 1)}(s) \biggr) \ph(s) ds \biggr)
}
for all $z \in i \R \setminus \{0\}$, where $a \ge 0$, $b, c \in \R$ and $\ph : \R \to \R$ satisfies the level crossing condition of Theorem~\ref{thm:bell}:
\begin{enumerate}[label=\textnormal{(\alph*)}]
\item\label{lem:bell:a} $\ph(s) \ge 0$ for $s > 0$, $\ph(s) \le 0$ for $s < 0$, and for all $k \in \Z$ the function $\ph(s) - k$ changes its sign at most once;
\item\label{lem:bell:b}
we have
\formula[eq:bell:b]{
 \biggl( \int_{-\infty}^{-1} + \int_1^\infty \biggr) \frac{|\ph(s)|}{|s|^3} \, ds < \infty ;
}
\item\label{lem:bell:c}
we have
\formula[eq:bell:c]{
 \int_{-1}^1 \re \laplace f(i t) dt & < \infty &\qquad& \text{and} &\qquad \lim_{t \to 0^+} (t \im \laplace f(i t)) & = 0.
}
\end{enumerate}
Conversely, if $a \ge 0$, $b, c \in \R$ and $\ph : \R \to \R$ satisfy conditions~\ref{lem:bell:a}, \ref{lem:bell:b} and~\ref{lem:bell:c} \textup{(}where $\laplace f(z)$ stands for the function defined by~\eqref{eq:bell}\textup{),} then there is a corresponding extended function $f$.
\end{lemma}

\begin{proof}
Suppose that $g \in \amcm$, $g$ is locally integrable and $g$ converges to zero at $\pm \infty$, and $h$ is a P\'olya frequency function. Then the convolution $f = g * h$ converges to zero at $\pm \infty$, and we have $\laplace f(z) = \laplace g(z) \laplace h(z)$ for all $z \in i \R \setminus \{0\}$. By Propositions~\ref{prop:amcm:laplace} and~\ref{prop:polya:laplace}, there are constants $\tilde{a} \ge 0$, $\tilde{b}, \tilde{c} \in \R$ and functions $\ph_1 : \R \to [-1, 1]$ and $\ph_2 : \R \to \Z$, satisfying appropriate conditions, such that for all $z \in i \R \setminus \{0\}$ we have
\formula{
 \laplace f(z) & = \exp\biggl(\tilde{a} z^2 - \tilde{b} z + \tilde{c} + \int_{-\infty}^\infty \biggl( \frac{1}{z + s} - \frac{1}{s} \ind_{\R \setminus (-1, 1)}(s)\biggr) \ph_1(s) ds \\
 & \hspace*{13.5em} + \int_{-\infty}^\infty \biggl( \frac{1}{z + s} - \frac{1}{s} + \frac{z}{s^2} \biggr) \ph_2(s) ds \biggr) .
}
This is equivalent to~\eqref{eq:bell} with $\ph(s) = \ph_1(s) + \ph_2(s)$ and $a = \tilde{a}$,
\formula[eq:bell:constants]{
\begin{aligned}
 b & = \tilde{b} + \biggl(\int_{-\infty}^{-1} + \int_1^\infty\biggr) \frac{1}{s^2} \, \ph_1(s) ds - \int_{-1}^1 \frac{1}{s^2} \, \ph_2(s) ds , \\
 c & = \tilde{c} - \int_{-1}^1 \frac{1}{s} \, \ph_2(s) ds .
\end{aligned}
}
Indeed, $\ph_1$ is bounded, and $\ph_2(s) = 0$ for $s$ in some neighbourhood of $0$, so that all integrals in~\eqref{eq:bell:constants} are convergent. It remains to verify that $\ph$ has all the desired properties.

Since both $\ph_1$ and $\ph_2$ are non-negative on $(0, \infty)$ and non-positive on $(-\infty, 0)$, $\ph$ has the same property. Integrability of $|\ph(s)| / |s|^3$ over $\R \setminus (-1, 1)$ follows from integrability of $|\ph_2(s)| / |s|^3$ and from $\ph_1$ being bounded. Furthermore, $\ph(s) \in [\ph_2(s), \ph_2(s) + 1]$ for $s > 0$, while $\ph(s) \in [\ph_2(s) - 1, \ph_2(s)]$ for $s < 0$, which easily implies that for $k \in \Z \setminus \{0\}$ the function $\ph(s) - k$ changes its sign at most once. Indeed, this change takes place at the argument $s$ such that $\ph_2(s^-) < k \le \ph_2(s^+)$ when $k > 0$, and at the argument $s$ such that $\ph_2(s^-) \le k < \ph_2(s^+)$ when $k < 0$ (and if such $s$ does not exists, then $\ph(s) - k$ has constant sign). When $k = 0$, $\ph(s) - k$ changes its sign at $s = 0$. Finally, property~\eqref{eq:bell:c} is satisfied, because $F(z) = \laplace g(z)$ satisfies condition~\eqref{eq:amcm:extra} in Proposition~\ref{prop:amcm:laplace} and $\laplace h$ is continuous on $i \R$.

We proceed with the proof of the converse implication. Suppose that $a$, $b$, $c$ and $\ph$ are as in the statement of the lemma. Thanks to the level-crossing condition, $\ph$ is locally bounded. This, combined with condition~\eqref{eq:bell:b}, implies that the integral in~\eqref{eq:bell} converges for every $z \in \R \setminus \{0\}$. We denote the right-hand side of~\eqref{eq:bell:b} by $F(z)$.

For $k \in \Z \setminus \{0\}$ the function $\ph(s) - k$ changes its sign at most once for all $k \in \Z \setminus \{0\}$; let us denote the argument at which this change occurs by $s_k$, namely,
\formula{
 s_k & = \sup \{ s \in \R : \ph(s) \le k \} &\qquad& \text{if $k > 0$,} \\
 s_k & = \inf \{ s \in \R : \ph(s) \ge k \} &\qquad& \text{if $k < 0$}
}
(possibly $s_k = \infty$ for some $k > 0$ and $s_k = -\infty$ for some $k < 0$). We also set $s_0 = 0$, and we let $\ph_2(s) = k$ for $s \in [s_k, s_{k+1})$ when $k \ge 0$ and $\ph_2(s) = k$ for $s \in [s_{k-1}, s_k)$ when $k \le 0$. Finally, we define $\ph_1(s) = \ph(s) - \ph_2(s)$.

We claim that $s_{-1} < 0$ and $s_1 > 0$. Indeed, suppose that $s_1 = 0$, so that $\ph(s) \ge 1$ for all $s > 0$. From the definition of $F(z)$ it follows that for $t \in (0, 1)$ we have
\formula{
 |F(i t)| & = \exp\biggl(-a t^2 + c + \int_{-\infty}^\infty \biggl( \frac{1}{t^2 + s^2} - \frac{1}{s^2} \ind_{\R \setminus (-1, 1)}(s) \biggr) s \ph(s) ds \biggr) \\
 & \ge \exp\biggl(-a t^2 + c + \biggl(\int_{-\infty}^{-1} + \int_1^\infty\biggr) \biggl( \frac{1}{1 + s^2} - \frac{1}{s^2} \biggr) s \ph(s) ds + \int_0^1 \frac{s}{t^2 + s^2} \, ds \biggr) ,
}
so that for some $C > 0$,
\formula{
 |F(i t)| & \ge C \exp\biggl(\int_0^1 \frac{s}{t^2 + s^2} \, ds \biggr) = C \sqrt{1 + \frac{1}{t^2}} \ge \frac{C}{t} \, .
}
This, however, contradicts~\eqref{eq:bell:c}: since $t \im F(i t)$ converges to zero as $t \to 0^+$, we necessarily have $\re F(i t) \ge C / (2 t)$ for $t$ in some right neighbourhood of $0$, and so the integral of $\re F(i t)$ divergent, contrary to~\eqref{eq:bell:c}. We conclude that $s_1 > 0$, and similarly we show that $s_1 < 0$.

It is easy to see that $\ph_1 : \R \to [-1, 1]$ and that $\ph_1$ is non-negative on $(0, \infty)$ and non-positive on $(-\infty, 0)$. From the definition of $\ph_2$ it also follows that $\ph_2$ maps $\R$ into $\Z$ and $\ph_2$ is non-decreasing. We already proved that $s_{-1} < 0$ and $s_1 > 0$, and thus $\ph_2(s) = 0$ for $s \in [s_{-1}, s_1)$. Hence, the function $|\ph_2(s)| / |s|^3$ is integrable over $(-1, 1)$. Since $|\ph_2(s)| \le |\ph(s)|$, the function $|\ph_2(s)| / |s|^3$ is also integrable over $\R \setminus (-1, 1)$.

With the above information, we let $\tilde{a} = a$ and we define $\tilde{b}$ and $\tilde{c}$ so that relations~\eqref{eq:bell:constants} are satisfied. We conclude that there is a P\'olya frequency function $h$ which corresponds to parameters $\tilde{a}$, $\tilde{b}$ and $\ph_2$ in representation~\eqref{eq:polya:exp}. Since $\laplace h$ is a continuous function on $i \R$ and $\laplace h(0) > 0$, the function $G(z) = F(z) / \laplace h(z)$ satisfies conditions~\eqref{eq:amcm:exp} and~\eqref{eq:amcm:extra} in Proposition~\ref{prop:amcm:laplace}, with parameters $\tilde{c}$ and $\ph_1$. Therefore, there is a locally integrable $\amcm$ extended function $g$ which converges to zero at $\pm \infty$ and such that $G(z) = \laplace g(z)$.

The extended function $f = g * h$ is the convolution of a locally integrable extended function $g \in \amcm$ which converges to zero at $\pm\infty$, and a Pólya frequency function~$h$. Furthermore, $\laplace f(z) = \laplace g(z) \laplace h(z) = F(z)$ for $z \in i \R \setminus \{0\}$. Thus, $f$ has all desired properties.
\end{proof}

\begin{proof}[Proof of Theorem~\ref{thm:bell}]
The desired result follows from Lemma~\ref{lem:bell} combined with Theorem~\ref{thm:amcm} and the fact that the convolution of a weakly bell-shaped function with a variation diminishing convolution kernel is again a weakly bell-shaped function.
\end{proof}

We remark that the function $\ph$ in Theorem~\ref{thm:bell} can be identified by studying the holomorphic extension of $\laplace f$. More precisely, $\log \laplace f(z)$ extends to a holomorphic function in the upper complex half-plane (this extension is given again by~\eqref{eq:bell}), and $\pi \ph(-s)$ is the boundary limit of the harmonic function $\arg \laplace f(z) = \im \log \laplace f(z)$ at $z = s$.

\begin{proof}[Proof of Corollary~\ref{cor:bell}]
Suppose that $a \ge 0$, $b, c \in \R$, $\nu : \R \setminus \{0\} \to \R$ is a function such that $x \nu(x)$ and $x \nu(-x)$ are completely monotone functions of $x > 0$, and in addition
\formula[eq:nu:x2]{
 \int_{-1}^1 x^2 \nu(x) dx < \infty .
}
Let $\pi_+$ and $\pi_-$ denote the Bernstein measures of the completely monotone functions $x \nu(x)$ and $x \nu(-x)$ (with $x > 0$), respectively. Finally, let $\ph_+(s) = \pi_+((0, s))$ and $\ph_-(s) = \pi_-((0, s))$ for $s > 0$. Integrating by parts, we obtain
\formula{
 \nu(x) & = \frac{1}{x} \int_{[0, \infty)} e^{-s x} \pi_+(ds) = \int_0^\infty e^{-s x} \ph_+(s) ds = \laplace \ph_+(x) ,
}
for $x > 0$, and similarly $\nu(x) = \laplace \ph_-(-x)$ for $x < 0$. Condition~\eqref{eq:nu:x2} translates easily into
\formula{
 \int_1^\infty \frac{\ph_+(s)}{s^3} \, ds & < \infty , & \int_1^\infty \frac{\ph_-(s)}{s^3} \, ds & < \infty . 
}
Let $\ph(s) = \ph_+(s)$ for $s > 0$, $\ph(s) = -\ph_-(-s)$ for $s < 0$ and $\ph(0) = 0$. Then $\ph$ is non-decreasing and it satisfies
\formula{
 \biggl( \int_{-\infty}^{-1} + \int_1^\infty \biggr) \frac{|\ph(s)|}{|s|^3} \, ds < \infty .
}
We find that for $z \in i \R \setminus \{0\}$, with the integrals with respect to $x$ understood as improper integrals, we have
\formula{
 & \int_{-\infty}^\infty (e^{-x z} - (1 + |x| - x z) e^{-|x|}) \nu(x) dx \\
 & \hspace*{2em} = \int_{-\infty}^0 (e^{-x z} - (1 - x - x z) e^x) \nu(x) dx + \int_0^\infty (e^{-x z} - (1 + x - x z) e^{-x}) \nu(x) dx \\
 & \hspace*{2em} = \int_{-\infty}^0 \biggl(-\frac{1}{z + s} - \frac{2 - s + z}{(-1 + s)^2} \biggr) \ph_-(s) ds + \int_0^\infty \biggl(\frac{1}{z + s} - \frac{2 + s - z}{(1 + s)^2}\biggr) \ph_+(s) ds \\
 & \hspace*{2em} = \int_{-\infty}^\infty \biggl(\frac{1}{z + s} - \frac{2 \sign s + s - z}{(1 + |s|)^2}\biggr) \ph(s) ds .
}
The above calculation involves simply Fubini's theorem when $\nu(x)$ is integrable near $\infty$, or, equivalently, if $|s^{-1} \ph(s)|$ is integrable near zero. In the general case the same argument works when $\re z > 0$ in the integral over $x > 0$ and when $\re z < 0$ in the integral over $x < 0$, and the desired identity follows by continuity of both integrals as $z$ approaches $i \R \setminus \{0\}$; we omit the details and refer to Lemma~\ref{lem:cm:reg} for a similar calculation.

It follows that
\formula{
 \laplace f(z) & = \exp\biggl(a z^2 - b z + c + \int_{-\infty}^\infty \biggl( \frac{1}{z + s} - \frac{2 \sign s + s - z}{(1 + |s|)^2} \biggr) \ph(s) ds \biggr) \\
 & = \exp\biggl(\tilde{a} z^2 - \tilde{b} z + \tilde{c} + \int_{-\infty}^\infty \biggl( \frac{1}{z + s} - \biggl(\frac{1}{s} - \frac{z}{s^2} \biggr) \ind_{\R \setminus (-1, 1)}(s) \biggr) \ph(s) ds \biggr) ,
}
where $\tilde{a} = a$,
\formula{
 \tilde{b} & = b + \int_{-\infty}^\infty \biggl( \frac{1}{s^2} \ind_{\R \setminus (-1, 1)}(s) - \frac{1}{(1 + |s|)^2} \biggr) \ph(s) ds , \\
 \tilde{c} & = c + \int_{-\infty}^\infty \biggl( \frac{1}{s} \ind_{\R \setminus (-1, 1)}(s) - \frac{2 \sign s + s}{(1 + |s|)^2}\biggr) \ph(s) ds .
}
Therefore, $\laplace f(z)$ satisfies~\eqref{eq:bell}. Finally, $\ph(0) = 0$ and $\ph$ is continuous at $0$, which easily implies that condition~\eqref{eq:bell:c} is satisfied.

We have thus proved that all assumptions of Theorem~\ref{thm:bell} are satisfied, and so its assertion applies to $f$.

Finally, the above calculations, combined with the final assertion of Theorem~\ref{thm:bell}, also show that any parameters $a$, $b$, $c$ and $\nu$ satisfying the assumptions listed in the statement of the corollary correspond to some function $f$ with the desired properties.
\end{proof}

\begin{proof}[Proof of Corollary~\ref{cor:stable}]
The density function of the L\'evy measure of a stable distribution is of the form
\formula{
 \nu(x) & = c_+ x^{-1 - \alpha} &\qquad& \text{for $x > 0$,} \\
 \nu(x) & = c_- (-x)^{-1 - \alpha} &\qquad& \text{for $x < 0$,}
}
where $c_+, c_- \ge 0$ and $\alpha \in (0, 2)$. Therefore, $x \nu(x) = c_+ x^{-\alpha}$ and $x \nu(-x) = c_- x^{-\alpha}$ are completely monotone functions of $x > 0$, and the desired result follows from Corollary~\ref{cor:bell}.
\end{proof}

\begin{remark}
Suppose that $g$ is a locally integrable extended function, possibly with a non-negative atom at $0$, and $g'$ is completely monotone on $(0, \infty)$ and $-g'$ is absolutely monotone on $(-\infty, 0)$. In Remark~\ref{rem:amcm:broad} we observed that this class of functions is weakly bell-shaped in the broad sense. It is easy to see that any such $g$ can be convolved with any P\'olya frequency function $h$ (because $|g|$ grows at most at a linear rate at $\pm\infty$, while $h$ has exponential decay at $\pm\infty$), and the convolution $f = g * h$ is again weakly bell-shaped in the broad sense.

Apparently, the Fourier transform of $g$ (defined, for example, in the sense of distributions) admits a description similar to that of Proposition~\ref{prop:amcm:laplace}. It would be interesting to derive an analogue of Lemma~\ref{lem:bell} for the class of convolutions $f = g * h$ discussed above.
\end{remark}

%
%

\section{Discussion}
\label{sec:ex}

The \emph{level crossing condition} in Theorem~\ref{thm:bell} requires $\ph(s) - k$ to change its sign at most once for $k \in \Z$. As discussed in the introduction, this assumption seems rather artificial, and it is natural to ask to what extent it can be relaxed. In this section we disprove three natural conjectures and discuss several examples. We used \emph{Wolfram Mathematica~10} computer algebra system for the (otherwise tiresome, but elementary) evaluation of some explicit expressions.

\subsection{Infinitely divisible distributions with $\amcm$ L\'evy measure}
\label{sec:ex:1}

It is natural to ask whether the level crossing condition~\ref{thm:bell:a} in Theorem~\ref{thm:bell} is completely superfluous, except for the assumption that $\ph(s) \ge 0$ for $s > 0$ and $\ph(s) \le 0$ for $s < 0$. With the notation of Corollary~\ref{cor:bell}, this modification amounts to relaxing complete monotonicity of $x \nu(x)$ and $x \nu(-x)$ for $x > 0$ to complete monotonicity of $\nu(x)$ and $\nu(-x)$ for $x > 0$. To see the equivalence of these extensions, recall that, for $x > 0$, $\nu(x)$ and $\nu(-x)$ are the Laplace transforms of $\ph(s) \ind_{(0, \infty)}(s)$ and $-\ph(-s) \ind_{(0, \infty)}(s)$, respectively. In yet another words: is Corollary~\ref{cor:bell} true for any $\nu \in \amcm$?

The answer is \emph{no}: we provide an example of an integrable function $f$ which is not weakly bell-shaped, which is in addition concentrated on $(0, \infty)$, and which satisfies all assumptions of Theorem~\ref{thm:bell} except the level crossing condition~\ref{thm:bell:a}. In other words, there is an infinitely divisible distribution on $(0, \infty)$ with L\'evy measure $\nu(x) \ind_{(0, \infty)}(x) dx$ for a completely monotone $\nu$, which is not bell-shaped.

\begin{figure}
\centering
\begin{tabular}{cc}
\includegraphics[width=0.4\textwidth]{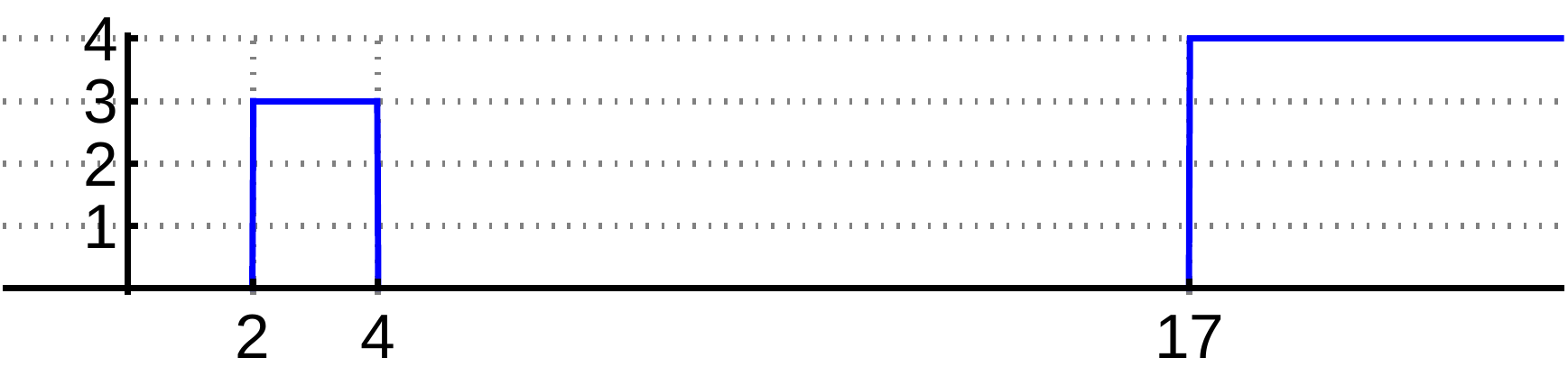} &
\includegraphics[width=0.4\textwidth]{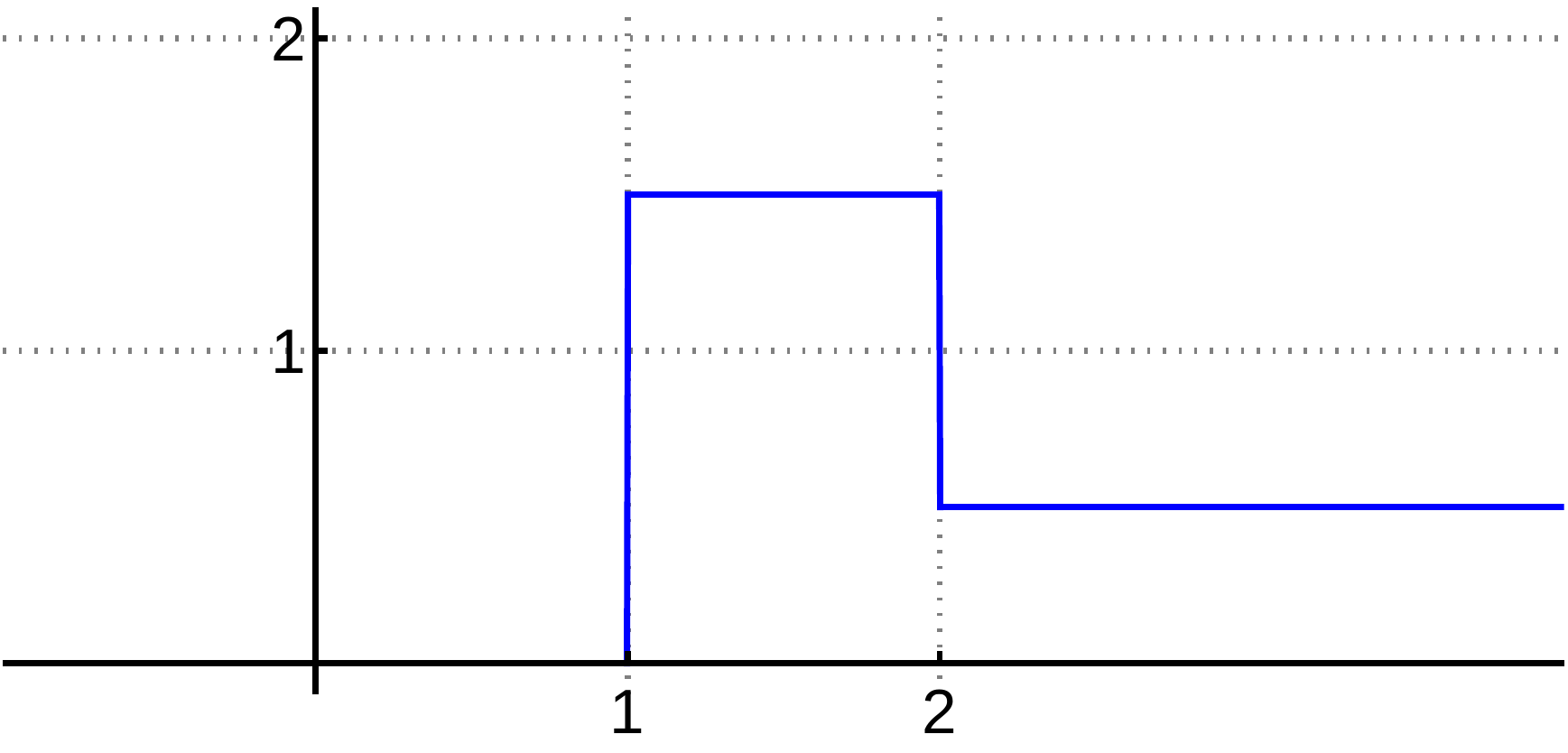} \\
(a) & (b)
\end{tabular}
\caption{Plot of the function $\ph(s)$ for the examples considered in: (a) Sections~\ref{sec:ex:1} and~\ref{sec:ex:2}; (b) Section~\ref{sec:ex:3}.}
\label{fig:ex}
\end{figure}

Consider a function $f : \R \to \R$ such that
\formula{
 \laplace f(z) & = \frac{(1 + z/4)^3}{(1 + z/2)^3 (1 + z/17)^4}
}
for $z \in i \R$. This corresponds to $a = c = 0$, $b = \tfrac{67}{68}$ and $\ph(s) = 3 \cdot \ind_{[2, 4)}(s) + 4 \cdot \ind_{[17, \infty)}(s)$ in Theorem~\ref{thm:bell} (see Figure~\ref{fig:ex}(a)), or $a = 0$ and $\nu(x) = x^{-1} (3 e^{-2 x} - 3 e^{-3 x} + 4 e^{-17 x}) \ind_{(0, \infty)}(x)$ in Corollary~\ref{cor:bell}. Clearly, the only assumption of Theorem~\ref{thm:bell} which is violated by $f$ is the level crossing condition: when $k = 1$ or $k = 2$, $\ph(s) - k$ changes its sign three times.

By inspecting the Laplace transform of $f$, it is easy to see that $f$ is constant zero on $(-\infty, 0)$, it is integrable with integral $1$, and it is twice continuously differentiable on $\R$. Furthermore, $\nu$ is positive on $(0, \infty)$, and so $f$ is the density function of an infinitely divisible distribution is positive on $(0, \infty)$ and $f(x) > 0$ for $x > 0$. It follows that in any neighbourhood of $0$ there is a number $x > 0$ such that $f''(x) > 0$, and there are arbitrarily large numbers $x > 0$ such that $f''(x) > 0$. Using \emph{Mathematica}, we find that
\formula{
 f(x) & = \tfrac{83521}{36450000} \bigl((284 + 888 x + 360 x^2) e^{-2 x} \\
 & \hspace*{10em} - (284 + 5148 x + 45630 x^2 - 494325 x^3) e^{-17 x}\bigr) ,
}
which leads to
\formula{
 f''(\tfrac{1}{4}) & = \tfrac{83521}{2332800000} (-38168123 e^{-17/4} - 92032 e^{-1/2}) < 0 , \\
 f''(\tfrac{1}{2}) & = \tfrac{83521}{2332800} (271849 e^{-17/2} - 64 e^{-1}) > 0 , \\
 f''(\tfrac{3}{4}) & = \tfrac{83521}{2332800000} (1787319463 e^{-51/4} - 24448 e^{-3/2}) < 0 .
}
Therefore, $f''$ changes its sign at least $4$ times (in fact --- exactly $4$ times). In particular, $f$ is not weakly bell-shaped.

The above function $f$ is not infinitely smooth. A similar example of a smooth function can be obtained as a convolution of $f$ with a smooth, sufficiently localised P\'olya frequency function $g$. We can choose $g$ to be the Gauss--Weierstrass kernel with appropriately small variance if we do not require $f * g$ to be concentrated on $(0, \infty)$. Otherwise, we need to choose $g$ to be concentrated on $(0, \infty)$ as well, and the expression for such a function $g$ is less explicit.

To be specific, consider $m$ sufficiently large, let $z_n = m 2^n$ for $n = 1, 2, \ldots$, and define $g_m$ to be the P\'olya frequency function with Laplace transform
\formula{
 \laplace g_m(z) & = \prod_{n = 1}^\infty \frac{1}{1 + z / z_n} \, .
}
With the notation of Definition~\ref{def:polya}, $g$ corresponds to $a = 0$, $b = \sum_{n = 1}^\infty z_n^{-1} = m^{-1}$, $N = \infty$ and $z_n$ as defined above. The distribution $g_m(x) dx$ is the convolution of infinitely many exponential measures $z_n \exp(-z_n x) \ind_{(0, \infty)}(x) dx$, so it is concentrated on $(0, \infty)$. The convolution of $n$ exponential distributions has $n - 1$ continuous derivatives, which easily implies that $g_m$ is smooth. Furthermore, $g_m(x) = m g_1(m x)$, so the family $g_m$ is an approximate identity. Finally, $g_m$ has the representation given in Theorem~\ref{thm:bell} and Corollary~\ref{cor:bell} with $a = c = 0$,
\formula{
 \ph_m(s) & = \sum_{j = 1}^\infty \ind_{[z_j, \infty)}(s)
}
and
\formula{
 \nu_m(x) & = \laplace \ph_+(x) = \sum_{j = 1}^\infty \frac{e^{-z_j x}}{x} \, .
}

The above observations imply that $f * g_m$ is a smooth function concentrated on $(0, \infty)$, and the density function $\nu(x) + \nu_m(x)$ of the L\'evy measure of $f * g_m$ (which is the sum of the densities of L\'evy measures of $f$ and $g$) is a completely monotone function on $(0, \infty)$. Additionally, $(f * g_m)'' = f'' * g_m$ converges pointwise to $f''$ as $m \to \infty$, so if $m$ is large enough, then $(f * g_m)''$ changes its sign more than twice. In particular, $f * g_m$ is not strictly bell-shaped.

\subsection{Self-decomposable distributions}
\label{sec:ex:2}

With the notation of Corollary~\ref{cor:bell}, $f$ is the density function of a \emph{positive self-decomposable distribution} if $f$ is positive on $(0, \infty)$, constant zero on $(-\infty, 0)$, it is integrable with integral $1$ and $x \nu(x)$ is a non-increasing function of $(0, \infty)$. T.~Simon stated two conjectures in~\cite{bib:s15}, which essentially state that density functions of positive self-decomposable distributions are weakly bell-shaped. This also turns out to be \emph{false}. More precisely, below we show that the functions $f$ and $f * g_m$ introduced in the previous section are counter-examples to both conjectures of T.~Simon.

Let $f$ be the function defined in the previous section. If $\nu(x) dx$ is the L\'evy measure of $f$, then
\formula{
 x \nu(x) & = 3 e^{-2 x} - 3 e^{-4 x} + 4 e^{-17 x}
}
for $x > 0$, which is easily shown to be decreasing: the function
\formula{
 2 e^{-2 x} (x \nu(x))' & = -3 + 6 e^{-2 x} - 34 e^{-15 x}
}
attains its global maximum at $x = \tfrac{1}{13} \log \tfrac{85}{2}$, and the maximal value is
\formula{
 -3 + \tfrac{26}{5} (\tfrac{2}{85})^{2/13} & < 0 .
}
It follows that $f$ is the density function of a positive self-decomposable distribution which is not weakly bell-shaped: in the previous section we proved that $f$ is twice continuously differentiable and $f''$ changes its sign more than twice. On the other hand, the limit of $x \nu(x)$ as $x \to 0$ is $4$. Conjecture~2 in~\cite{bib:s15} asserts that if the limit of $x \nu(x)$ as $x \to 0$ is greater than $n + 1$, then $f^{(n)}$ changes its sign $n$ times. This claim is clearly false for $n = 2$.

As in the previous section, the convolution of $f$ with a smooth, sufficiently localised P\'olya frequency function $g_m$ concentrated on $(0, \infty)$ provides an example of an smooth density function of a positive self-decomposable distribution which is not strictly bell-shaped. The density function of the L\'evy measure of $f * g_m$ is $\nu(x) + \nu_m(x)$ (the sum of the densities of L\'evy measures of $f$ and $g_m$), so that $x (\nu(x) + \nu_m(x))$ diverges to infinity as $x \to 0^+$. Therefore, Conjecture~1 in~\cite{bib:s15} would imply that $f$ is strictly bell-shaped, which we already know to be false.

\subsection{Functions corresponding to $\ph$ uniformly close to non-decreasing functions}
\label{sec:ex:3}

The function $\ph$ in Theorem~\ref{thm:bell} can only decrease by at most $1$ compared to its previous maximal value. For this reason it is natural to conjecture that Theorem~\ref{thm:bell} remains valid if we relax the the level crossing condition~\ref{thm:bell:a} to the following one: $\ph(s_2) - \ph(s_1) \ge -1$ when $s_1 \le s_2$. However, also this conjecture turns out to be \emph{false}, as shown by the following simple example. Note that the argument used below can be easily adapted to show that an even stronger condition: $\ph(s_2) - \ph(s_1) \ge -\eps$ when $s_1 \le s_2$, with arbitrarily small $\eps > 0$, is not sufficient.

Let $f : \R \to \R$ be defined by
\formula{
 \laplace f(z) & = \frac{1 + z/2}{(1 + z)^{3/2}}
}
for $z \in i \R$. This corresponds to $a = c = 0$, $b = 1$ and $\ph(s) = \tfrac{3}{2} \ind_{[1, 2)}(s) + \tfrac{1}{2} \ind_{[2, \infty)}(s)$ in Theorem~\ref{thm:bell} (see Figure~\ref{fig:ex}(b)), and $f$ is easily found to be equal to
\formula{
 f(x) & = \frac{1}{2 \sqrt{\pi}} \, \frac{1 + 2 x}{\sqrt{x}} \, e^{-x} \ind_{(0, \infty)}(x) \, .
}
We claim that for sufficiently small $t > 0$, the convolution $f * G_t$ of $f$ and the Gauss--Weierstrass kernel $G_t$ is not bell-shaped: $(f * G_t)^{(8)}$ changes its sign more than $8$ times.

Let $g(x) = (2 \sqrt{\pi})^{-1} x^{-1/2} \ind_{(0, \infty)}(x)$ and $h(x) = (g * G_1)^{(8)}(x)$. Since $g$ is completely monotone on $(0, \infty)$ and locally integrable, $g$ is weakly bell-shaped, and therefore $h$~changes its sign $8$ times. Furthermore, $g^{(8)}$ is positive on $(0, \infty)$, so $h$ is positive near $\infty$. If follows that there is an increasing sequence $x_1, x_2, \ldots, x_9$ such that $(-1)^{j - 1} h(x_j) > 0$ for $j = 1, 2, \ldots, 9$.

Note that $f - g$ is a bounded function. Therefore, with the usual notation $\|f\|_1 = \int_{-\infty}^\infty |f(x)| dx$ and $\|f\|_\infty = \sup \{|f(x)| : x \in \R\}$, we have
\formula{
 |(f * G_t)^{(8)}(x) - (g * G_t)^{(8)}(x)| & = |(f - g) * G_t^{(8)}(x)| \\
 & \le \|f - g\|_\infty \|G_t^{(8)}\|_1 = t^{-4} \|f - g\|_\infty \|G_1^{(8)}\|_1
}
for all $x \in \R$. Furthermore, $g * G_t(x) = t^{-1/4} g * G_1(t^{-1/2} x)$, so $g * G_t^{(8)}(x) = t^{-17/4} h(t^{-1/2} x)$. It follows that
\formula{
 |t^{17/4} (f * G_t)^{(8)}(t^{1/2} x) - h(x)| & \le t^{1/4} \|f - g\|_\infty \|G_1^{(8)}\|_1
}
for all $x \in \R$. By considering $x = t^{-1/2} x_j$, we conclude that if $t > 0$ is sufficiently small, we have $(-1)^{j - 1} (f * G_t)^{(8)}(t^{1/2} x_j) > 0$ for $j = 1, 2, \ldots, 9$.

From the exponential representation of Stieltjes functions (see~\cite{bib:ad56,bib:ssv12}) and the expression for $\laplace f$ it follows that $f$ is not completely monotone on $(0, \infty)$, and indeed using \emph{Mathematica} we easily find that $f^{(8)}(4) = -\tfrac{11598375}{67108864} e^{-4} \pi^{-1/2} < 0$. Therefore, $(f * G_t)^{(8)}(4) < 0$ for sufficiently small $t > 0$.

It follows that for $t > 0$ sufficiently small we have $(-1)^{j - 1} (f * G_t)^{(8)}(t^{1/2} x_j) > 0$ for $j = 1, 2, \ldots, 9$, $(f * G_t)^{(8)}(4) < 0$ and $t^{1/2} x_9 < 4$. However, this means that $(f * G_t)^{(8)}$ changes its sign at least $9$ times, and so $f$ is not weakly bell-shaped.

\subsection{Previously known classes of bell-shaped functions}

The following classes of functions that have been previously shown to be bell-shaped are included in Theorem~\ref{thm:bell}:
\begin{enumerate}[label=\textnormal{(\alph*)}]
\item The Gauss--Weierstrass kernel $G_t$: the $n$-th derivative of $G_t$ is equal to $G_t$ multiplied by the Hermite polynomial of degree $n$, and thus it has exactly $n$ zeroes. In Theorem~\ref{thm:bell}, $G_t$ corresponds to $a = t$, $b = c = 0$ and $\ph(s) = 0$.
\item Functions $f_p(x) = (1 + x^2)^{-p}$, where $p > 0$, are bell-shaped, because $(1 + x^2)^{p + n} f_p^{(n)}$ is a polynomial of degree $n$ for $n = 0, 1, 2, \ldots$\, The Fourier transform of $f_p$ is given by
\formula{
 \laplace f_p(z) & = c_p |z|^{p-1/2} K_{p-1/2}(|z|)
}
for $z \in i \R$, and it corresponds to $a = 0$ and
\formula{
 \ph(s) & = \tfrac{1}{\pi} \arg(i J_{p-1/2}(|s|) - Y_{p-1/2}(|s|)) \sign s
}
in Theorem~\ref{thm:bell}. Here $c_p > 0$ is a constant, $K_\nu$, $J_\nu$ and $Y_\nu$ denote appropriate Bessel functions and $\arg$ stands for the continuous version of the complex argument of a zero-free function, determined uniquely by the condition $|\ph(s)| < \pi$ for $s$ in a neighbourhood of $0$. The above expression for $\ph$ can be derived by finding the boundary limit of the bounded harmonic function $\arg \laplace f$ in the upper complex half-plane, using well-known properties of Bessel functions; we refer to Section~\ref{sec:main} and to the proof of Theorem~1 in~\cite{bib:i90} for additional details.
\item Functions $f_p(x) = x^{-p} e^{-1/x} \ind_{(0, \infty)}(x)$ are bell-shaped for $p > 0$: the function $e^{1/x} x^{-p - 2 n} f_p^{(n)}(x)$ is equal on $(0, \infty)$ to a polynomial of degree $n$. Since $f_p$ is the density function of a hitting time of a diffusion (discussed below), it is indeed included in Theorem~\ref{thm:bell}; we refer to~\cite{bib:js15} for further details.
\item P\'olya frequency functions are bell-shaped; in Theorem~\ref{thm:bell} they correspond to non-decreasing functions $\ph$ that only take integer values.
\item Density functions of all stable distributions concentrated on $(0, \infty)$, considered in~\cite{bib:s15}.
\item Density functions of hitting times of generalised diffusions, studied in~\cite{bib:js15}. These are indeed included in Theorem~\ref{thm:bell}, because they can be represented as convolutions of P\'olya frequency functions and completely monotone functions on $(0, \infty)$, as discussed in~\cite{bib:js15}.
\end{enumerate}

\subsection{An explicit example}

Let $0 < p < q$ and
\formula{
 f(x) & = \frac{1}{\pi} \, \frac{p q (p + q)}{(p^2 + x^2) (q^2 + x^2)} \, .
}
By a simple calculation, one finds that the Fourier transform of $f$ has the representation given in Theorem~\ref{thm:bell}, with $a = b = c = 0$ and an increasing function $\ph$. Indeed, $\ph(s)$ is the continuous version of $\arg(q e^{i p s} - p e^{i q s})$. By Theorem~\ref{thm:bell}, $f$ is strictly bell-shaped. The author is not aware of any elementary proof of this fact.

Interestingly, the function $f(x) = 210 \pi^{-1} (1 + x^2)^{-1} (9 + x^2)^{-1} (16 + x^2)^{-1}$ does not satisfy the assumptions of Theorem~\ref{thm:bell}, and it is not bell-shaped: as can be explicitly checked with the help of \emph{Mathematica}, $f^{(57)}$ changes its sign $61$ times.

%
%


\begin{thebibliography}{00}

\bibitem{bib:ad56}
N.~Aronszajn, W.~F.~Donoghue,
\emph{On exponential representations of analytic functions in the upper half-plane with positive imaginary part}.
J.~Anal. Math. 5 (1956): 321--385.

\bibitem{bib:be06}
W.~Bergweiler, A.~Eremenko,
\emph{Proof of a conjecture of P\'olya on the zeros of successive derivatives of real entire functions}.
Acta Math. 197(2) (2006): 145--166.

\bibitem{bib:b92}
L.~Bondesson,
\emph{Generalized Gamma Convolutions and Related Classes of Distributions and Densities}.
Lecture Notes in Statistics 76, Springer-Verlag, New York, 1992.

\bibitem{bib:mo}
A.~Eremenko, \emph{Characterisation of bell-shaped functions}.
MathOverflow, available at \url{https://mathoverflow.net/q/282680} (2017).

\bibitem{bib:f08}
S.~Fisk,
\emph{Polynomials, roots, and interlacing}.
Preprint, arXiv:math/0612833v2, 2008.

\bibitem{bib:g84}
W.~Gawronski,
\emph{On the bell-shape of stable densities}.
Ann. Probab. 12(1) (1984): 230--242.

\bibitem{bib:h50}
I.~I.~Hirschman,
\emph{Proof of a conjecture of I.~J.~Schoenberg}.
Proc. Amer. Math. Soc.~1 (1950): 63--65.

\bibitem{bib:i90}
M.~E.~H.~Ismail,
\emph{Complete Monotonicity of Modified Bessel Functions}.
Proc. Amer. Math. Soc. 108(2) (1990): 353--361.

\bibitem{bib:js15}
W.~Jedidi, T.~Simon,
\emph{Diffusion hitting times and the bell-shape}.
Stat. Probab. Lett. 102 (2015): 38--41.

\bibitem{bib:k68}
S.~Karlin,
\emph{Total positivity. Vol. 1}.
Stanford University Press, Stanford, CA, 1968.

\bibitem{bib:kk00}
H.~Ki, Y.-O.~Kim,
\emph{On the number of nonreal zeros of real entire functions and the Fourier--P\'olya conjecture}.
Duke Math. J. 104(1) (2000): 45--73.

\bibitem{bib:k19}
M.~Kwa\'snicki,
\emph{Fluctuation theory for L\'evy processes with completely monotone jumps}.
Preprint (2018), arXiv:1811.06617.

\bibitem{bib:p43}
G.~P\'olya,
\emph{On the zeros of the derivatives of a function and its analytic character}.
Bull. Amer. Math. Soc. 49 (1943): 178--191.

\bibitem{bib:r83}
L.C.G.~Rogers,
\emph{Wiener--Hopf factorization of diffusions and Lévy processes}.
Proc. London Math. Soc. 47(3) (1983): 177--191.

\bibitem{bib:s99}
K.~Sato,
\emph{Lévy Processes and Infinitely Divisible Distributions}.
Cambridge Univ. Press, Cambridge, 1999.

\bibitem{bib:ssv12}
R.~Schilling, R.~Song, Z.~Vondra{\v{c}}ek,
\emph{Bernstein Functions: Theory and Applications}.
Studies in Math. 37, De Gruyter, Berlin, 2012.

\bibitem{bib:s47}
I.~J.~Schoenberg,
\emph{On Totally Positive Functions, Laplace Integrals and Entire Functions of the Laguerre--Polya--Schur Type}.
Proc. Nat. Acad. Sci. 33(1) (1947): 11--17.

\bibitem{bib:s48}
I.~J.~Schoenberg,
\emph{On Variation-Diminishing Integral Operators of the Convolution Type}.
Proc. Nat. Acad. Sci. 34(4) (1948): 164--169.

\bibitem{bib:s15}
T.~Simon,
\emph{Positive stable densities and the bell-shape}.
Proc. Amer. Math. Soc. 143(2) (2015): 885--895.

\bibitem{bib:hw55}
D.~V.~Widder, I.~I.~Hirschman,
\emph{The Convolution Transform}.
Princeton Math. Ser. 20, Princeton University Press, Princeton, 1955.

\end{thebibliography}
\end{document}